\documentclass[reqno]{amsart}
\usepackage{relsize}
\usepackage{scalerel}
\usepackage{stackengine,wasysym}
\usepackage{todonotes}
\usepackage{xcolor}
\usepackage{mathrsfs}
\usepackage{dsfont}
\usepackage{mathtools}
\usepackage{hyperref}
\usepackage[sort,nocompress]{cite}
\usepackage{amsmath,amssymb} 
\usepackage[toc,page]{appendix}
\usepackage{bbm}
\usepackage{verbatim}
\allowdisplaybreaks
\theoremstyle{plain}
\newtheorem{lemma}{Lemma}[section]
\newtheorem{theorem}[lemma]{Theorem}
\newtheorem{corollary}[lemma]{Corollary}

\newtheorem{proposition}[lemma]{Proposition}
\newtheorem{definition}[lemma]{Definition}

\theoremstyle{remark}
\newtheorem{remark}{Remark}

\newcommand*  {\N} {{\mathbb N}}

\newcommand*  {\Z} {{\mathbb Z}}

\newcommand{\Nb}{\mathbb{N}}

\newcommand{\Eb}{\mathbb{E}}
\newcommand{\Pb}{\mathbb{P}}

\newcommand{\Fb}{\mathbb{F}}
\newcommand{\Fc}{\mathcal{F}}
\newcommand{\Fct}{\left( \mathcal{F}_t \right)_{t \geq 0}}

\newcommand{\Uc}{\mathscr{U}}
\newcommand{\Xc}{\mathcal{X}}

\newcommand{\Sc}{\mathcal{S}}

\newcommand{\bx}{\boldsymbol{x}}
\newcommand{\bk}{\boldsymbol{k}}

\newcommand{\hook}{\hookrightarrow}	

\makeatletter
\newcommand*{\rom}[1]{\expandafter\@slowromancap\romannumeral #1@}
\makeatother

\numberwithin{equation}{section}
\begin{document}

\vskip 0.125in

\title[Stochastic Primitive Equations]
{Local martingale solutions and pathwise uniqueness for the Three-dimensional stochastic inviscid primitive equations}

\date{\today}
%\thanks{\textit{ }}

\author[R. Hu]{Ruimeng Hu}
\address[R. Hu]
{	Department of Mathematics \\
Department of Statistics and Applied Probability \\
     University of California  \\
	Santa Barbara, CA 93106, USA.} \email{rhu@ucsb.edu}

\author[Q. Lin]{Quyuan Lin*}\thanks{*Corresponding author. Department of Mathematics, University of California,	Santa Barbara, CA 93106, USA. E-mail address: quyuan\_lin@ucsb.edu}
\address[Q. Lin]
{	Department of Mathematics \\
     University of California  \\
	Santa Barbara, CA 93106, USA.} \email{quyuan\_lin@ucsb.edu}

\begin{abstract}
We study the stochastic effect on the three-dimensional inviscid primitive equations (PEs, also called the hydrostatic Euler equations). Specifically, we consider a larger class of noises than multiplicative noises, and work in the analytic function space due to the ill-posedness in Sobolev spaces of PEs without horizontal viscosity. Under proper conditions, we prove the local existence of martingale solutions and pathwise uniqueness. By adding vertical viscosity, {\it i.e.}, considering the hydrostatic Navier-Stokes equations, we can relax the restriction on initial conditions to be only analytic in the horizontal variables with Sobolev regularity in the vertical variable, and allow the transport noise in the vertical direction. We establish the local existence of martingale solutions and pathwise uniqueness, and show that the solutions become analytic in the vertical variable instantaneously as $t>0$ and the vertical analytic radius increases as long as the solutions exist. 
\end{abstract}

\maketitle

MSC Subject Classifications: 35Q86, 60H15, 76M35, 35Q35, 86A10\\

Keywords: stochastic primitive equations; local martingale solutions; pathwise uniqueness; hydrostatic Euler equations; hydrostatic Navier-Stokes equations

\section{Introduction} 
In this paper, we consider the following three-dimensional stochastic primitive equations (PEs), which serves as one of the fundamental models in geophysical fluid dynamics, \begin{subequations}\label{PE-system}
\begin{align}
    &d V + (V\cdot \nabla V + w\partial_z V - \nu_h \Delta V - \nu_z \partial_{zz} V +f_0 V^\perp + \nabla p + f) dt = \sigma (V) dW , \label{PE-1}  
    \\
    &\partial_z p  =0 , \label{PE-2}
    \\
    &\nabla \cdot V + \partial_z w =0,  \label{PE-3} 
\end{align}
\end{subequations}
where the horizontal velocity field $V=(u, v)$, the vertical velocity $w$, and the pressure $p$ are the unknown quantities. The non-negative constants $\nu_h$ and $\nu_z$ are horizontal and vertical viscosities, and we denote by $\nabla = (\partial_x, \partial_y)$ 
and $\Delta = \partial_{xx} + \partial_{yy}$. The parameter $f_0 \in \mathbb{R}$ stands for the speed and direction of rotation in the Coriolis force, and $V^\perp = (-v, u)$. The term $\sigma (V) dW$ represents the external forcing term driven by white noise, and $f$ is the other external force. The three-dimensional (3D) viscous PEs are derived by performing a formal asymptotic limit of the small aspect ratio (the ratio of the depth or the height to the horizontal length scale) from the Rayleigh-B\'enard (Boussinesq) system. This limit is justified rigorously in \cite{azerad2001mathematical,li2019primitive,li2022primitive}. Notice that we have omitted the coupling with the temperature and salinity in \eqref{PE-system} to focus our attention towards difficulties arising from the ill-posedness in Sobolev spaces when $\nu_h=0$ \cite{renardy2009ill}.

To address the well-posedness of \eqref{PE-system} in the inviscid case, {\it i.e.}, $\nu_h=\nu_z=0$, one needs to overcome the following difficulties. We shall also discuss some key issues of our work below and explain how the restrictions can be relaxed when considering only vertical viscosity. 
\begin{enumerate}
    \item The first difficulty is that, even in the deterministic case, the PEs with $\nu_h=0$ is ill-posed in Sobolev spaces and Gevrey class of order $s>1$. For this reason, one has to work in the analytic function space. To the best of our knowledge, very little work has been done for stochastic partial differential equations in the analytic function space, as most of the results work in the classical Sobolev spaces.
    
    \item Unlike the stochastic PEs with $\nu_h>0$ where the global existence of solutions can be proven, we can only show the local existence of martingale solutions. The main reason is that, in the deterministic case, the solutions to the PEs with $\nu_h=0$ are unknown to exist globally when $\nu_z>0$, and have been shown to form singularity in finite time when $\nu_z=0$ \cite{cao2015finite,collot2021stable,ibrahim2021finite,wong2015blowup}. 
    
    \item Compared to the viscosity which can give ``strong dissipation'' by providing one gain in the derivative, the analytic framework only gives ``weak dissipation'' by providing one-half gain in the derivative. Therefore, we can not consider transport noise in any direction for the inviscid case, which requires one gain in the derivative. When we consider the vertical viscosity in \eqref{PE-system}, we can allow the existence of vertical transport noise. On the other hand, thanks to the ``weak dissipation'' effect from the analytic framework, one can consider the noises in a class that is larger than multiplicative noises; see Remark~\ref{remark:noise} for more details.
    
    \item In \cite{brzezniak2021well,debussche2011local,saal2021stochastic} where $\nu_h >0$, the existence of both martingale solutions (weak solution in the stochastic sense) and pathwise solutions (strong solution in the stochastic sense) has been shown. The authors first showed the existence of martingale solutions and the pathwise uniqueness to a modified system, then used a Yamada-Watanabe type argument (see \cite[Proposition 2.2]{debussche2011local} or \cite{gyongy1996existence}) and applied a suitable stopping time to establish the existence of pathwise solutions to the original system. In fact, the ``strong dissipation'' from the horizontal viscosity can help control the nonlinear estimates. In contrast, the ``weak dissipation'' from the analytic framework is not strong enough to provide the same control on the nonlinear estimates in our case. Consequently, we are only able to show the pathwise uniqueness to the original system, but not to the modified system. This prevents us from showing the existence of pathwise solutions; see Propositions~\ref{prop:pathwise.uniqueness} and \ref{prop:pathwise.uniqueness-viscous} and Remark~\ref{remark:nonlinear-trouble} for more details.
    
    \item For the inviscid PEs, the initial condition needs to be analytic in all directions, and the analytic radii of the solutions in all variables shrink. In terms of the case with vertical viscosity, the initial condition only needs to be analytic in the horizontal variables with Sobolev regularity in the $z$ variable. Moreover, the solutions become analytic in $z$ variable instantaneous when $t>0$, and the analytic radius in the $z$ variable increases with time as long as the solutions exist (see \eqref{tau-gamma}). This result is parallel to the deterministic case \cite{Lin2022effect}, and is due to the effect of vertical viscosity.
\end{enumerate}

System \eqref{PE-system} is usually studied under some physical boundary conditions. For example, one considers the physical domain to be $\mathbb T^2 \times [0,H]$ and
\begin{equation}\label{BC-physical}
    V \text{ is periodic in }  (x,y) \text{ with period }  1,  \qquad (\partial_z V,w)|_{z=0,H} = 0.
\end{equation}
Here the condition on $\partial_z V$ vanishes when $\nu_z = 0$. 

In this paper, instead of $\mathbb T^2 \times [0,H]$, we consider the domain to be $\mathbb T^3$, and
\begin{equation}\label{BC-T3}
    V \text{ is periodic in }  (x,y,z) \text{ with period }  1,  \qquad V  \text{ is even in }  z  \text{ and }  w  \text{ is odd in }  z.
\end{equation}
Observe that the space of periodic functions with such symmetry condition is invariant under the dynamics of system \eqref{PE-system}. If $H=\frac{1}{2}$, a solution to system \eqref{PE-system} in $\mathbb{T}^3$ subject to \eqref{BC-T3} restricted on the horizontal channel $\mathbb T^2 \times [0,H]$ is a solution to system \eqref{PE-system} subject to the physical boundary conditions \eqref{BC-physical}. The same simplification has been done in \cite{ghoul2022effect}. Working in $\mathbb{T}^3$ allows us to use Fourier analysis, and makes the mathematical presentation simpler and more elegant.

\noindent {\bf Main results}. The following two theorems are the main results of this paper, which concern the local existence of martingale solutions and pathwise uniqueness of the 3D stochastic hydrostatic Euler equations \eqref{PE-inviscid-system-abstract} and 3D stochastic hydrostatic Navier-Stokes equations \eqref{PE-viscous-system-abstract}.

\begin{theorem}[Stochastic hydrostatic Euler]\label{theorem:inviscid}
Suppose that $\mu_0$ satisfies \eqref{condition:mu-zero} with constant $M>0$. 
Let $\rho \geq M$, $\tau_0>0$, and $r>\frac52$ be fixed. Assume and that the noise $\sigma$ and the external force $f$ satisfy \eqref{noise-inviscid} and \eqref{force-inviscid}, respectively. Then there exist a time $T$, a function $\tau(t)$, and a stopping time $\eta$, defined in \eqref{T}, \eqref{equation:tau}, and \eqref{stopingtime:eta}, respectively, such that there exists a local martingale solution $(\Sc, W, V, \eta\wedge T)$ to system \eqref{PE-inviscid-system-abstract} in the sense of Definition~\ref{definition:inviscid-solution}. Moreover, system \eqref{PE-inviscid-system-abstract} also has pathwise uniqueness. 
\end{theorem}

With the help of the vertical viscosity in the stochastic hydrostatic Navier-Stokes equations, we can relax the restriction on initial conditions to be only analytic in the horizontal variables with Sobolev regularity in the vertical variable, and show similar local existence of martingale solutions and pathwise uniqueness in the following theorem.

\begin{theorem}[Stochastic hydrostatic Navier-Stokes]\label{theorem:viscous}
Suppose that $\mu_0$ satisfies \eqref{condition:mu-zero-viscous} with constant $M>0$. 
Let $\rho \geq M$, $\tau_0>0$, and $r>\frac52$ be fixed, and let $\gamma_0 = 0$. Assume and that the noise $\sigma$ and the external force $f$ satisfy \eqref{noise-viscous} and \eqref{force-viscous}, respectively. Then there exist a time $T$, functions $\tau(t)$ and $\gamma(t)$, and a stopping time $\eta$, defined in \eqref{T-viscous}, \eqref{tau-gamma}, and \eqref{stopingtime:eta-viscous}, respectively, such that there exists a local martingale solution $(\Sc, W, V, \eta\wedge T)$ to system \eqref{PE-viscous-system-abstract} in the sense of Definition~\ref{definition:viscous-solution}. Moreover, system \eqref{PE-viscous-system-abstract} also has pathwise uniqueness.
\end{theorem}

\noindent{\bf Related literature}. For the deterministic case, the global well-posedness of strong solutions to the 3D PEs with full viscosity was first established in \cite{cao2007global}, and later in \cite{kobelkov2006existence,kukavica2007regularity,hieber2016global}. In \cite{cao2016global,cao2017strong,cao2020global}, the authors showed the global well-posedness of strong solutions to the 3D PEs with only horizontal viscosity. When $\nu_h = 0$, the inviscid PEs ({\it i.e.}, $\nu_z=0$) are called the hydrostatic Euler equations, and with only vertical viscosity ({\it i.e.}, $\nu_z>0$) are called the hydrostatic Navier-Stokes equations. It has been shown that the PEs with $\nu_h=0$ are linearly ill-posed in any Sobolev spaces and Gevrey class of order $s>1$ \cite{renardy2009ill} (see also \cite{han2016ill,ibrahim2021finite}). With only vertical viscosity, to overcome the ill-posedness, one can consider additional weak dissipation \cite{cao2020well}, assume Gevrey regularity with some convex condition \cite{gerard2020well}, or take the initial data to be analytic in the horizontal direction and only Sobolev in the vertical direction without any special structure \cite{paicu2020hydrostatic,Lin2022effect}. For the inviscid case, the ill-posedness can be overcomed by assuming either some special structures (local Rayleigh condition) on the initial data or real analyticity in all directions for general initial data \cite{brenier1999homogeneous,brenier2003remarks,ghoul2022effect,grenier1999derivation,kukavica2011local,kukavica2014local,masmoudi2012h}. While the strong solutions to the PEs with the horizontal viscosity $\nu_h>0$ exist globally in time, it has been shown that smooth solutions to the inviscid PEs can develop singularity in finite time \cite{cao2015finite,collot2021stable,ibrahim2021finite,wong2015blowup}. Whether the smooth solutions to the PEs with only vertical viscosity exist globally or blow up in finite time still remains open.

{On the other hand,} stochastic factors in primitive equations are essential in studying geophysical fluid dynamics. For example, introducing white noise terms to the system could account for numerical and empirical uncertainties. It can also provide probabilistic predictions, which is a range of possible scenarios associated with their likelihoods.  Along this direction, the stochastic PEs with full viscosity were studied by \cite{glatt2008stochastic,glatt2011pathwise} in the 2D case, and by \cite{glatt2008stochastic,brzezniak2021well,debussche2011local,debussche2012global} in the 3D case. In the recent work \cite{brzezniak2021well}, the global well-posedness of strong solutions (strong in both PDE and stochastic sense) was established with multiplicative and transport noise, under smallness assumption on the transport noise. The global existence of solutions is based on the deterministic result. With only horizontal viscosity, global existence and uniqueness of strong solutions have been established in \cite{saal2021stochastic}, where the noise can be transported but only in the horizontal direction, and the global existence of solutions is also based on the deterministic result. The two situations above ($\nu_h>0$) allow the classical Sobolev framework as the PEs with horizontal viscosity are well-posed in Sobolev spaces in the deterministic case.

The rest of the paper is organized as follows. In Section~\ref{section:preliminary}, we introduce the functional settings, the notations, the assumptions, and some preliminary results that will be used throughout the paper. In Section~\ref{section:inviscid}, we analyze the stochastic hydrostatic Euler equations ({\it i.e.}, the inviscid case with  $\nu_h=\nu_z=0$ in \eqref{PE-system}), and prove Theorem~\ref{theorem:inviscid}. Section~\ref{section:viscous} focuses on the stochastic hydrostatic Navier-Stokes equations ({\it i.e.}, the vertically viscous case with $\nu_h=0$ and $\nu_z>0$ in \eqref{PE-system}) where we prove Theorem~\ref{theorem:viscous}. We make conclusive remarks in Section~\ref{sec:conclusion} and provide the detailed nonlinear estimates in the appendix. 

%%%%%%%%%%%%%%%%%%%%%%%%%%%%%%%%%
%%%%%%%%%%%%%%%%%%%%%%%%%%%%%%
%%%%%%%%%%%%%%%%%%%%%%%%%%%%%
\section{Preliminaries}\label{section:preliminary}

The universal constant $C$ appears in the paper may change from line to line. We shall use subscripts to indicate the dependence of $C$ on other parameters, {\it e.g.}, $C_r$ means that the constant $C$ depends only on $r$. 

\subsection{Functional settings} 

Let $\bx:= (\bx',z) = (x_1, x_2, z)\in \mathbb{T}^3$, where $\bx'$ and $z$ represent the horizontal and vertical variables, respectively, and $\mathbb{T}^3$ denotes the three-dimensional torus with unit length. Denote by
$$
    \|f \|:=\|f\|_{L^2(\mathbb{T}^3)} = (\int_{\mathbb{T}^3} |f(\bx)|^2 d\bx)^{\frac{1}{2}},
$$
associated with the inner product
$  \langle f,g\rangle = \int_{\mathbb{T}^3} f(\bx)g(\bx) d\bx
$
for $f,g \in L^2(\mathbb{T}^3)$. For a function $f \in L^2(\mathbb{T}^3)$, let $\hat{f}_{\bk}$ be its Fourier coefficient such that
\begin{equation*}
    f(\bx) = \sum\limits_{\bk\in 2\pi\mathbb{Z}^3} \hat{f}_{\bk} e^{ i\bk\cdot \bx}, \qquad \hat{f}_{\bk} = \int_{\mathbb{T}^3} e^{- i\bk\cdot \bx} f(\bx) d\bx.
\end{equation*}
For $r\geq 0$, we define the following Sobolev $H^r$ norm and $\dot{H}^r$ semi-norm
\begin{eqnarray*}
&&\hskip-.1in
\|f \|_{H^r}:= \Big(\sum\limits_{\bk\in 2\pi\mathbb{Z}^3} (1+|\bk|^{2r}) |\hat{f}_{\bk}|^2  \Big)^{\frac{1}{2}}, \qquad \|f \|_{\dot{H}^r}:= \Big(\sum\limits_{\bk\in  2\pi\mathbb{Z}^3} |\bk|^{2r} |\hat{f}_{\bk}|^2  \Big)^{\frac{1}{2}},
\end{eqnarray*}
and we refer to \cite{adams2003sobolev} for more details about Sobolev spaces.

The divergence-free condition \eqref{PE-3} and the oddness of $w$ in $z$ imply that
\begin{equation*}
    \int_0^1 \nabla \cdot V (\bx',z)dz = 0.
\end{equation*}
Moreover, we consider $V$ even in the $z$ variable. That is, $V\in \mathcal D_0$ where
\begin{equation*}
    \mathcal D_0 := \left\{ f\in L^2(\mathbb{T}^3) : \int_0^1 \nabla \cdot f (\bx',z)dz = 0 ,\, f \text{   is even in } z\right\}.
\end{equation*}

\subsubsection{Inviscid case}
Let $\tau(t) \geq 0$ represent the radius of analyticity, and let $A = \sqrt{-(\Delta + \partial_{zz})}$ and $e^{\tau(t)A}$ be defined by, in terms of the Fourier coefficients, 
\begin{equation*}
    \widehat{(Af)}_{\bk}  := |\bk| \hat{f}_{\bk}, \ \ (\widehat{e^{\tau(t)A}f})_{\bk}  := e^{\tau(t)|\bk|} \hat{f}_{\bk}, \qquad \boldsymbol k \in 2\pi \mathbb Z^3. 
\end{equation*}
In this paper, the argument of $\tau(t)$ is frequently omitted if no confusion arises. 
Then for $\tau, r\geq 0$, $A^rf$ and $e^{\tau A} f$ are given by
\begin{equation*}
    A^rf := \sum\limits_{\bk\in  2\pi\mathbb{Z}^3} |\bk|^r \hat{f}_{\bk} e^{ i \bk\cdot\bx}, \ \ e^{\tau A} f := \sum\limits_{\bk\in  2\pi\mathbb{Z}^3} e^{\tau|\bk|} \hat{f}_{\bk} e^{ i \bk\cdot\bx}.
\end{equation*}
For $r\geq 0$, we define a family of normed spaces, parameterized by $\tau\geq 0$,
\begin{eqnarray*}
\mathcal{D}_{\tau,r} := \{f\in \mathcal D_0: \|f\|_{\tau,r} <\infty  \},
\end{eqnarray*}
where the equipped norm is 
\begin{eqnarray*}
\|f\|_{\tau,r}:=\|e^{\tau A} f\|_{H^r} = \Big(\sum\limits_{\bk\in  2\pi\mathbb{Z}^3} (1+|\bk|^{2r}) e^{2\tau |\bk} |\hat{f}_{\bk}|^2 \Big)^{\frac{1}{2}}, \label{analytic-norm}
\end{eqnarray*}
and the corresponding semi-norm is 
\begin{eqnarray*}
\|A^r e^{\tau A} f\| = \|e^{\tau A} f\|_{\dot{H}^r} =  \Big(\sum\limits_{\bk\in  2\pi\mathbb{Z}^3} |\bk|^{2r} e^{2\tau |\bk|} |\hat{f}_{\bk}|^2\Big)^{\frac{1}{2}}. \label{anlytic-semi-norm}
\end{eqnarray*}
It is easy to see that
\begin{eqnarray}\label{inviscid-norm-equi}
\|A^r e^{\tau A} f\|^2 + \|f\|^2 \leq \| f\|_{\tau,r}^2 \leq  2(\|A^r e^{\tau A} f\|^2 + \|f\|^2).
\end{eqnarray}
Notice that when $\tau=0$, one has $\mathcal D_{0,r} = H^r \cap \mathcal  D_0.$
\begin{remark}
The space of analytic functions is defined as ({\it cf.} \cite{levermore1997analyticity})
\begin{eqnarray*}
G^1(\mathbb{T}^3) = \bigcup\limits_{\tau >0} \mathcal{D}_{\tau,r},
\end{eqnarray*}
which is a special case of the Gevrey class \cite{ferrari1998gevrey,foias1989gevrey,levermore1997analyticity}. 
Since inviscid PEs are ill-posed in Sobolev spaces and in the Gevrey class $G^s$ of order $s>1$ \cite{ibrahim2021finite,renardy2009ill,han2016ill}, we shall focus on the Gevrey class of order $s=1$, which is the space of analytic function. 
\end{remark}

\subsubsection{Vertically viscous case}
For the vertically viscous case, we consider $V\in \mathcal{D}_{\tau,r,\gamma,s}$ where
\begin{equation*}
    \mathcal{D}_{\tau,r,\gamma,s} := \Big\{ f\in \mathcal D_0, \, \|f\|_{\tau,r,\gamma,s} < \infty \Big\}.
\end{equation*}
The equipped norm is defined by
\begin{gather*}
   \|f\|_{\tau,r,\gamma,s} := \left(\sum\limits_{\boldsymbol{k}'\in  2\pi \mathbb{Z}^2, k_3 \in 2\pi \mathbb Z} \left(1 + |\boldsymbol{k}'|^{2r} + |k_3|^{2s}\right) e^{2\tau |\boldsymbol{k}'|} e^{2\gamma |k_3|} |\hat{f}_{\boldsymbol{k}',k_3}|^2 \right)^{\frac12},
   \end{gather*}
where we denote by 
$$
\boldsymbol{k} = (\boldsymbol{k}',k_3), \qquad \hat{f}_{\boldsymbol{k}} \equiv \hat{f}_{\boldsymbol{k}',k_3} :=\int_{\mathbb T^3} e^{-i\boldsymbol{k}' \cdot \boldsymbol{x}'- ik_3 z} f(\boldsymbol{x}',z)\, d\boldsymbol{x}' dz.
$$
Here $\tau\geq 0$ and $\gamma \geq 0$ represent the horizontal and vertical radius of analyticity, respectively. Notice that when $\tau=\gamma=0$, $\mathcal D_{0,r,0,s} = H_{\boldsymbol{x}'}^r \cap H_z^s \cap \mathcal D_0.$ 
 For $ \bk = (\boldsymbol k' ,k_3)  \in 2\pi (\mathbb Z^2 \times \mathbb Z), ~ \tau,\gamma\geq 0$, let $A_h = \sqrt{-\Delta}$, $A_z=\sqrt{-\partial_{zz}}$, and $e^{\tau A_h}e^{\gamma A_z}$ be defined by, in terms of the Fourier coefficients,
\begin{equation*}
    (\widehat{A_h f})_{\bk} := |\boldsymbol{k}'| \hat{f}_{\bk}, \quad (\widehat{A_z f})_{\bk} := |k_3| \hat{f}_{\bk}, \quad (\widehat{e^{\tau A_h}e^{\gamma A_z} f})_{\bk} = e^{\tau |\boldsymbol{k}'|} e^{\gamma |k_3|} \hat{f}_{\bk}. 
\end{equation*}
Then for $\tau, \gamma, r, s\geq 0$, one has
\begin{equation*}
A_h^r f := \sum\limits_{\bk\in  2\pi\mathbb{Z}^3} |\bk'|^r\hat{f}_{\bk} e^{ i \bk\cdot\bx}, \quad 
A_z^s  f := \sum\limits_{\bk\in  2\pi\mathbb{Z}^3} |k_3|^s  \hat{f}_{\bk} e^{ i \bk\cdot\bx}, \quad
    e^{\tau A_h} e^{\gamma A_z} f := \sum\limits_{\bk\in  2\pi\mathbb{Z}^3}  e^{\tau |\boldsymbol{k}'|} e^{\gamma |k_3|} \hat{f}_{\bk} e^{ i \bk\cdot\bx}.
\end{equation*} 
In particular, we are interested in the case when $s=r$. 
Notice that when $\tau=\gamma=0$ one has $\mathcal D_{0,r,0,r} = \mathcal D_{0,r}= H^r\cap \mathcal D_0.$ Moreover, we have 
\begin{equation*}
\begin{split}
    \| f\|^2 +  \|A^r e^{\tau A_h} e^{\gamma A_z} f\|^2 \leq \|f\|_{\tau,r,\gamma,r}^2 & \leq 2(\| f\|^2 +  \|A^r e^{\tau A_h} e^{\gamma A_z} f\|^2) .
\end{split}
\end{equation*}

\subsection{Assumptions on noise and force terms}\label{assump:noise}
\subsubsection{Inviscid case}
For the inviscid case, for any $\tau\geq0$ and $r\geq0$, let $\sigma: \mathcal D_{\tau,r+\frac12} \rightarrow L_2(\Uc, \mathcal D_{\tau,r})$, and assume that the noise $\sigma$ satisfies the growth conditions and the Lipschitz continuity:
\begin{equation}\label{noise-inviscid}
\begin{split}
    \|\sigma(V)\|^2_{L_2(\Uc, \mathcal D_{\tau,r})} \leq C \Big(1+ \| V\|^2_{\tau,r+\frac12} \Big)&, \\
    \| \sigma(V) - \sigma(V^\#)\|^2_{L_2(\Uc, \mathcal D_{\tau,r})} \leq C  \| V - V^\#\|^2_{\tau,r+\frac12}&,  \qquad V, V^\# \in \mathcal D_{\tau,r+\frac12}.
\end{split}
\end{equation}
where $L_2(X, Y)$ denotes the space of Hilbert-Schmidt operators from $X$ to $Y$ and we defer the precise definition of $\Uc$ to Section~\ref{sec:stochasticprelim}. 

For the force $f$, with $\tau_0>0$ and $r>\frac52$, we assume that
\begin{equation}\label{force-inviscid}
    f \in L^2(0,\infty; \mathcal D_{\tau_0,r}).
\end{equation}

\subsubsection{Vertically viscous case}
For the vertically viscous case, for any $\tau\geq0$, $\gamma\geq 0$, and $r\geq 0$, let $\sigma: \mathcal D_{\tau,r+\frac12,\gamma,r+1} \rightarrow L_2(\Uc, \mathcal D_{\tau,r,\gamma,r})$, and assume that the noise $\sigma$ satisfies the growth condition and the Lipschitz continuity:
\begin{equation}\label{noise-viscous}
\begin{split}
    \|\sigma(V)\|^2_{L_2(\Uc, \mathcal D_{\tau,r+\frac12,\gamma,r+1})} \leq C \Big(1+ \| V\|^2_{\tau,r+\frac12,\gamma,r} \Big) &+ \delta^2  \| V\|^2_{\tau,r+\frac12,\gamma,r+1} ,
    \\
    \| \sigma(V) - \sigma(V^\#)\|^2_{L_2(\Uc, \mathcal D_{\tau,r+\frac12,\gamma,r+1})}
    \leq C  \| V - V^\#\|^2_{\tau,r+\frac12,\gamma,r}  & + \delta^2 \| V - V^\#\|^2_{\tau,r+\frac12,\gamma,r+1},
    \\
    &V, V^\# \in \mathcal D_{\tau,r+\frac12,\gamma,r+1} ,
\end{split}
\end{equation}
with $\delta$ small enough depending on $\nu_z$, as indicated in Lemma \ref{lemma:estimate-viscous} and Proposition  \ref{prop:pathwise.uniqueness-viscous}. 

For the force $f$, with $\tau_0>0$, $\gamma^* \geq \frac{\nu_z\tau_0}8$, and $r>\frac52$, we assume that
\begin{equation}\label{force-viscous}
    f \in L^2(0,\infty; \mathcal D_{\tau_0,r,\gamma^*,r}).
\end{equation}
Notice that $f$ has large enough analytic radius in $z$, which will allow the growth of analytic radius in $z$ for the solutions.

\begin{remark}\label{remark:noise}
 The assumptions on the noise for the inviscid case allow a class of noise that is larger than the set of the multiplicative noises, but does not contain the transport noise. On the other hand, the assumptions for the vertically viscous case allow the transport noise in the vertical direction. In general, if the system has the viscosity in some directions, then the transport noise in that direction is allowed since the viscosity can control the loss of the derivative. See, for example, \cite{brzezniak2021well} where the transport noise in all spatial directions is allowed since they consider full viscosity, and \cite{saal2021stochastic} where the authors consider the transport noise only in the horizontal direction since they only have horizontal viscosity. In our inviscid case, it is therefore natural to exclude transport noise. However, thanks to the framework of analytic functions, we are allowed to consider a class of noise larger than the set of multiplicative noises.
\end{remark}

\subsection{Stochastic preliminaries}\label{sec:stochasticprelim}
Let $\Sc = \left(\Omega, \Fc, \Fb, \Pb\right)$ be a stochastic basis with filtration $\Fb = \Fct$.  Let $\Uc$ be a separable Hilbert space and let $W$ be an $\Fb$-adapted cylindrical Wiener process with reproducing kernel Hilbert space $\Uc$ on $\Sc$. Let $\lbrace e_k \rbrace_{k = 1}^\infty$ be an orthonormal basis of $\Uc$, then $W$ may be formally written as $W = \sum_{k=1}^\infty e_kW^k$, where $W^k$ are independent 1D Wiener processes on $\Sc$.

Let $X$ be another separable Hilbert space and denote by $L_2(\Uc, X)$ the collection of Hilber-Schmidt operators from $\Uc$ into $X$. Given a predictable process $\Phi \in L^2\bigl(\Omega; L^2\left(0, T; L_2\left(\Uc, X \right)\right)\bigr)$, one may define the stochastic integral with respect to the cylindrical Wiener process by
\[
	\int_0^T \Phi \, dW = \sum_{k = 1}^\infty \int_0^T \Phi e_k \, dW^k.
\]
Note that the integral can also be extended to $\Phi$ with $\int_0^T \| \Phi \|^2_{L_2\left( \Uc, X \right)} \, dt < \infty$, $\mathbb{P}$-almost surely, and we refer readers to \cite[Section 4]{da2014stochastic} for more details. 

Let us also recall the definitions of Sobolev spaces with fractional time derivative, see {\it e.g.}\ \cite{simon1990sobolev}. Let $X$ be a separable Hilbert space and let $t > 0$, $p > 1$ and $\alpha \in (0, 1)$. We define
\[
	W^{\alpha, p}(0, t; X) =  \left\{ u \in L^p(0, t; X) \mid \int_0^t \int_0^t \frac{\vert u(s) - u(r) \vert_X^p}{|s - r|^{1+\alpha p}} \, dr \, ds < \infty \right\}
\]
and equip it with the norm
\[
	\| u \|^{p}_{W^{\alpha, p}(0, t; X)} = \int_0^t \vert u(s) \vert^p_X \, ds + \int_0^t \int_0^t \frac{\vert u(s) - u(r) \vert_X^p}{|s - r|^{1+\alpha p}} \, dr \, ds.
\]

In the sequel,  we shall repeated use two versions of the Burkholder-Davis-Gundy inequality. For $\Phi \in L^2\left(\Omega; L^2\left(0, T; L_2\left(\Uc, X\right)\right)\right)$, one has
\begin{equation}
	\label{eq:bdg}
	\Eb \sup_{t \in \left[0, T\right]} \left| \int_0^t \Phi \, dW \right|^r_X \leq C_{r} \, \Eb \left( \int_0^T \| \Phi \|_{L_2(\Uc, X)}^2 \, dt \right)^{r/2}.
\end{equation}
Moreover, if $p \geq 2$ and $\Phi \in L^p\left(\Omega; L^p\left(0, T; L_2\left(\Uc, X\right)\right)\right)$, then
\begin{equation}
	\label{eq:bdg.frac}
	\Eb \left| \int_0^\cdot \Phi \, dW \right|^p_{W^{\alpha, p}(0, T; X)} \leq c_{p} \, \Eb \int_0^T \| \Phi \|_{L_2(\Uc, X)}^p \, dt,
\end{equation}
for $\alpha \in [0, 1/2)$. For proofs, see, for instance, \cite{karatzas2012brownian}  and \cite[Lemma 2.1]{flandoli1995martingale}.

\subsection{Notion of solution}

In this paper, we only consider martingale solutions ({\it i.e.}, weak solutions in the stochastic sense), and adapt the definition from \cite{debussche2011local,brzezniak2021well}.
\subsubsection{Inviscid case}
First, we consider the following inviscid version of system \eqref{PE-system}, 
\begin{subequations}\label{PE-inviscid-system}
\begin{align}
    &d V + (V\cdot \nabla V + w\partial_z V  +f_0 V^\perp + \nabla p + f)dt = \sigma (V) dW , \label{PE-inviscid-1}  
    \\
    &\partial_z p  =0 , \label{PE-inviscid-2}
    \\
    &\nabla \cdot V + \partial_z w =0,  \label{PE-inviscid-3} 
    \\
    &V(0) = V_0. \label{PE-inviscid-ic}
\end{align}
\end{subequations}
From the system above, one can write $P(\bx')=p(\bx',z)$ and $w=-\int_0^z \nabla \cdot V(\bx',\tilde z)d\tilde z$. Denoting by
\begin{equation*}
    Q(g,h) = g\cdot \nabla h - \left(\int_0^z (\nabla \cdot  g) (\tilde z)  d\tilde z \right) \partial_z h, \ \ F(g) :=  f_0 g^{\perp} + \nabla P +  f,
\end{equation*}
we can rewrite \eqref{PE-inviscid-system} as
\begin{subequations}\label{PE-inviscid-system-abstract}
\begin{align}
    &d V + \big[Q( V, V) + F(V) \big]dt = \sigma (V) dW , \label{PE-inviscid-abstract-1}  
    \\
    &V(0) = V_0. \label{PE-inviscid-abstract-ic}
\end{align}
\end{subequations}

We consider the initial data to be given by a Borel measure $\mu_0$ on $\mathcal D_{\tau_0,r}$ such that for some $M\geq 0$,
\begin{equation}\label{condition:mu-zero}
    \mu_0\left(\left\{ V\in \mathcal D_{\tau_0,r}, \Vert V  \Vert_{\tau_0,r} \geq M \right\} \right) = 0.
\end{equation}
In particular, this implies that for any 
$p \geq 1$,
\begin{equation}
	\label{eq:mu.zero}
	\int_{\mathcal D_{\tau_0,r}} \Vert V  \Vert_{\tau_0,r}^p \, d\mu_0(V) < \infty.
\end{equation}
\begin{remark}\label{remark:initial-condition}
 The assumption on the initial condition \eqref{condition:mu-zero} is to guarantee the stopping time \eqref{stopingtime:eta} appearing in Corollary \ref{cor:loc.mart.sol} to be positive almost surely. One could also follow the strategy in \cite{brzezniak2021well,debussche2011local} to weaken such assumption by considering an additional linear differential equation in the modified system \eqref{PE-inviscid-system-modified}. In order to simplify our presentation and focus on the difficulties discussed in Section 1, we take the assumption \eqref{condition:mu-zero} in our paper.
\end{remark}

\begin{definition}\label{definition:inviscid-solution}
Let $r>\frac52$, $T>0$, $\tau(t)\geq 0$ for $t\in[0,T]$, and let $\mu_0$ satisfy \eqref{condition:mu-zero} with some constant $M>0$. Assume that $\sigma$ and $f$ satisfy \eqref{noise-inviscid} and \eqref{force-inviscid}, respectively. We call a quadruple $(\Sc, W, V, \eta\wedge T)$ a \emph{local martingale solution} if $\Sc = \left(\Omega, \Fc, \Fb, \Pb\right)$ is a stochastic basis, $W$ is an $\Fb$-adapted cylindrical Wiener process with reproducing kernel Hilbert space $\Uc$, $\eta$ is an $\Fb$-stopping time and $V\left(\cdot \wedge \eta\right): \Omega \times [0, T] \to \mathcal D_{\tau(t),r}$ is a progressively measurable process such that $\eta > 0$ $\Pb$-a.s.,
\begin{equation*}
	V\left( \cdot \wedge \eta \right) \in L^2\left(\Omega; C\left([0, T], \mathcal D_{\tau(t),r}\right)\right), \qquad \mathds{1}_{[0, \eta]}(\cdot) V \in L^2\left(\Omega; L^2\left(0, T; \mathcal D_{\tau(t),r+\frac12} \right)\right),
\end{equation*}
the law of $V(0)$ is $\mu_0$ and $V$ satisfies the following equality in $\mathcal D_0$ for all $t \in[0,T]$
\begin{equation*}
	V\left(t \wedge \eta\right) + \int_0^{t \wedge \eta} \big[Q(V,V)+F(V)\big] ds = V(0) + \int_0^{t \wedge \eta} \sigma(V) \, dW.
\end{equation*}

\end{definition}

\subsubsection{Vertically viscous case}
For the vertically viscous version of system \eqref{PE-system}, we consider
\begin{subequations}\label{PE-viscous-system-abstract}
\begin{align}
    &d V + \big[Q( V, V) + F(V) - \nu_z \partial_{zz} V\big]dt  = \sigma (V) dW , \label{PE-viscous-1}  
    \\
    &V(0) = V_0, \label{PE-viscous-ic}
\end{align}
\end{subequations}
We consider the initial data to be given by a Borel measure $\mu_0$ on $\mathcal D_{\tau_0,r,0,r}$ such that for some $M>0$,
\begin{equation}\label{condition:mu-zero-viscous}
    \mu_0\left(\left\{ V\in \mathcal D_{\tau_0,r,0,r}, \Vert V  \Vert_{\tau_0,r,0,r} \geq M \right\} \right) = 0.
\end{equation}
In particular, this implies that for any $p \geq 1$,
\begin{equation}
	\label{eq:mu.zero-viscous}
	\int_{\mathcal D_{\tau_0,r,0,r}} \Vert V  \Vert_{\tau_0,r,0,r}^p \, d\mu_0(V) < \infty.
\end{equation}

\begin{definition}\label{definition:viscous-solution}
Let $r>\frac52$, $T>0$, $\tau(t),\gamma(t)\geq 0$ for $t\in[0,T]$, and let $\mu_0$ satisfy \eqref{condition:mu-zero-viscous} with some constant $M>0$. Assume that $\sigma$ and $f$ satisfy \eqref{noise-viscous} and \eqref{force-viscous}, respectively. We call a quadruple $(\Sc, W, V, \eta\wedge T)$ a \emph{local martingale solution} if $\Sc = \left(\Omega, \Fc, \Fb, \Pb\right)$ is a stochastic basis, $W$ is an $\Fb$-adapted cylindrical Wiener process with reproducing kernel Hilbert space $\Uc$, $\eta$ is an $\Fb$-stopping time and $V\left(\cdot \wedge \eta\right): \Omega \times [0, T] \to \mathcal D_{\tau(t),r,\gamma(t),r}$ is a progressively measurable process such that $\eta > 0$ $\Pb$-a.s.,
\begin{equation*}
	\begin{split}
	    	&V\left( \cdot \wedge \eta \right) \in L^2\left(\Omega; C\left([0, T], \mathcal D_{\tau(t),r,\gamma(t),r}\right)\right), 
	    	\\
	    	&\mathds{1}_{[0, \eta]}(\cdot) A^r  V \in L^2\left(\Omega; L^2\left(0, T; \mathcal D_{\tau(t),\frac12,\gamma(t),1}\right)\right),
	\end{split}
\end{equation*}
the law of $V(0)$ is $\mu_0$ and $V$ satisfies the following equality in $\mathcal D_0$ for all $t \in[0,T]$
\begin{equation*}
	V\left(t \wedge \eta\right) + \int_0^{t \wedge \eta} \big[Q(V,V)+F(V) - \nu_z \partial_{zz} V\big] ds = V(0) + \int_0^{t \wedge \eta} \sigma(V) \, dW.
\end{equation*}

\end{definition}

\subsection{Compactness theorem}\label{sec:compactness}

We recall the following compactness result which are needed for our results. For proofs see \cite[Theorem 5]{simon1986compact} and \cite[Theorem 2.1]{flandoli1995martingale}, respectively.

\begin{lemma} 
\label{lemma:aubin-lions}
a) \emph{(Aubin-Lions-Simon Lemma).} Let $X_2 \subset X \subset X_1$ be Banach spaces such that the embedding $X_2 \hook \hook X$ is compact and the embedding $X \hook X_1$ is continuous. Let $p \in (1, \infty)$ and $\alpha \in (0, 1)$. Then the following embedding is compact
\[
	L^p(0, t; X_2) \cap W^{\alpha, p}(0, t; X_1) \hook \hook L^p(0, t; X).
\]

b) Let $X_2 \subset X$ be Banach spaces such that $X_2$ is reflexive and the embedding $X_2 \hook \hook X$ is compact. Let $\alpha \in (0, 1]$ and $p \in (1, \infty)$ be such that $\alpha p > 1$. Then the following embedding is compact
\[
	W^{\alpha, p}(0, t; X_2) \hook \hook C([0, t], X).	
\]
\end{lemma}

\section{The Inviscid case}\label{section:inviscid}
In this section, we focus on system \eqref{PE-inviscid-system-abstract}. In order to deal with the nonlinear terms in the energy estimates, we use the strategy introduced in \cite{debussche2011local,brzezniak2021well,saal2021stochastic} and consider the following modified system
\begin{subequations}\label{PE-inviscid-system-modified}
\begin{align}
    &d V + \big[\theta_\rho(\|V\|_{\tau,r})Q( V, V) + F(V) \big]dt = \sigma (V) dW , \label{PE-inviscid-modified-1}  
    \\
    &V(0) = V_0, \label{PE-inviscid-modified-ic}
\end{align}
\end{subequations}
where $\rho\geq M$ is a fixed constant with $M$ appearing in the condition \eqref{condition:mu-zero}, and the function $\theta_\rho(x)\in C^\infty(\mathbb R)$ is a non-increasing cut-off function such that 
\begin{equation}\label{eqn:rho}
    \mathbbm{1}_{[-\frac{\rho}2, \frac{\rho}2]} \leq \theta_\rho(x) \leq \mathbbm{1}_{[-\rho, \rho]}.
\end{equation}

The main goal of this section is to prove Theorem~\ref{theorem:inviscid}. For this purpose, we fix $\rho \geq M$, $\tau_0>0$ and $r>\frac52$ through this section. In Sections~\ref{section:galerkin} and \ref{section:energy-estimate-inviscid}, we first work on a Galerkin approximation system of the modified system \eqref{PE-inviscid-system-modified} and derive some necessary energy estimates. In Section~\ref{section:existence-inviscid}, by Prokhorov's Theorem and Skorokhod Theorem, the existence of local martingale solution to the modified system \eqref{PE-inviscid-system-modified} is established. By defining a suitable stopping time, we show the existence of local martingale solutions to the original system \eqref{PE-inviscid-system-abstract}. Finally, in Section~\ref{section:uniqueness-inviscid}, we establish the pathwise uniqueness to the original system \eqref{PE-inviscid-system-abstract}, which completes the proof of Theorem~\ref{theorem:inviscid}.

\subsection{Galerkin scheme}\label{section:galerkin}
We employ the Galerkin approximation procedure. For $\bk\in 2\pi\mathbb{Z}^3$, let
\begin{equation*}
\phi_{\bk} = \phi_{k_1,k_2,k_3} :=
\begin{cases}
\sqrt{2}e^{ i\left( k_1 x_1 + k_2 x_2 \right)}\cos( k_3 z) & \text{if} \;  k_3\neq 0,\\
e^{ i\left( k_1 x_1 + k_2 x_2 \right)} & \text{if} \;  k_3=0,  \label{phik}
\end{cases}
\end{equation*}
and
\begin{eqnarray*}
	&&\hskip-.28in
	\mathcal{E}:=  \{ \phi \in L^2(\mathbb{T}^3) \; | \; \phi= \sum\limits_{\bk\in 2\pi\mathbb{Z}^3} a_{\bk} \phi_{\bk}, \; a_{-k_1, -k_2,k_3}=a_{k_1,k_2,k_3}^{*}, \; \sum\limits_{\bk\in 2\pi\mathbb{Z}^3} |a_{\bk}|^2 < \infty \},
\end{eqnarray*}
where $a^*$ denotes the complex conjugate of $a$. Notice that $\mathcal{E}$ is a closed subspace of $L^2(\mathbb{T}^3)$, and consists of real-valued functions which are even in $z$ variable. For any $n\in \N$, denote by
\begin{eqnarray*}
	&&\hskip-.28in
	\mathcal{E}_n:=  \{ \phi \in L^2(\mathbb{T}^3)\; | \; \phi= \sum\limits_{|\bk|\leq n} a_{\bk} \phi_{\bk}, \; a_{-k_1, -k_2, k_3}=a_{k_1,k_2, k_3}^{*}  \},
\end{eqnarray*}
the finite-dimensional subspaces of $\mathcal{E}$. For any function $f\in L^2(\mathbb{T}^3)$, denote by 
\begin{equation*}
f_{\bk} :=\int_{\mathbb{T}^3} f(\bx)\phi^*_{\bk}(\bx) d\bx,
\end{equation*}
and write $P_n f := \sum_{|\bk|\leq n} f_{\bk} \phi_{\bk}$ for $n\in \N$. Then $P_n$ is the orthogonal projections from $L^2(\mathbb{T}^3)$ to $\mathcal{E}_n$. Moreover, one has the following Poincar\'e inequality:
\begin{equation}\label{poincare}
    \|(I-P_n)f\|^2 = \sum\limits_{|\bk|>n} |f_{\bk}|^2 \leq \frac1{n} \sum\limits_{|\bk|>n} |\bk| |f_{\bk}|^2 =  \frac1{n} \|(I-P_n)f\|_{0,\frac12}^2.
\end{equation}
Denote by 
\begin{equation*}
    Q_n(g , h) :=  P_n Q(g,h), \ \ F_n(g) := P_n F(g), \ \ \sigma_n := P_n \sigma.
\end{equation*}
For each $n\in \N$, we consider the following Galerkin approximation of system \eqref{PE-inviscid-system-modified} at order $n$ as:
\begin{subequations}\label{PE-inviscid-system-galerkin}
\begin{align}
    &d V_n + \big[\theta_\rho(\|V_n\|_{\tau,r} ) Q_n(V_n, V_n) + F_n(V_n) \big]dt   = \sigma_n (V_n) dW , \label{PE-inviscid-galerkin-1}  
    \\
    & V_n(0) = P_n V_0 =: (V_n)_0. \label{PE-inviscid-galerkin-ic}  
\end{align}
\end{subequations}
Since all the terms in system \eqref{PE-inviscid-system-galerkin} are locally Lipschitz, the local existence and uniqueness of the solution $V_n$ is a standard result. 

For some function $\tau(t)$ to be determined later, with initial condition denote by $\tau_0$, {\it i.e.}, $\tau(0) = \tau_0$, we define $U_n = e^{\tau A}V_n$, and therefore, $V_n = e^{-\tau A}U_n$. Since $dU_n = \dot{\tau}A U_n dt + e^{\tau A}dV_n$, system \eqref{PE-inviscid-system-galerkin} is equivalent to
\begin{subequations}\label{PE-inviscid-system-galerkin-U}
\begin{align}
    &d U_n + \big[\theta_\rho(\|U_n\|_{0,r} ) e^{\tau A} Q_n(e^{-\tau A}U_n, e^{-\tau A}U_n) + e^{\tau A}F_n(e^{-\tau A}U_n) -\dot{\tau} A U_n\big]dt   = e^{\tau A}\sigma_n (e^{-\tau A}U_n) dW , \label{PE-inviscid-galerkin-U-1}  
    \\
    & U_n(0) = e^{\tau_0 A }(V_n)_0 = P_n \left(e^{\tau_0 A } V_0\right), \label{PE-inviscid-galerkin-U-ic}  
\end{align}
\end{subequations}
and $U_n = e^{\tau A}V_n$ is the solution to system \eqref{PE-inviscid-system-galerkin-U}.

\subsection{Energy estimates} \label{section:energy-estimate-inviscid}
Now we establish the main estimates needed to pass to the limit in the Galerkin system \eqref{PE-inviscid-system-galerkin} (and equivalently, system \eqref{PE-inviscid-system-galerkin-U}).

\begin{lemma}\label{lemma:estimate-inviscid}
Let $\tau_0>0$, $p\geq 2$, and $r>\frac52$ be fixed. Suppose that $V_0 \in L^p(\Omega; \mathcal D_{\tau_0,r})$ is an $\Fc_0$-measurable random variable. Suppose that $\sigma$ and $f$ satisfy \eqref{noise-inviscid} and \eqref{force-inviscid}, respectively.
Let $V_n$ be the solution to system \eqref{PE-inviscid-system-galerkin}. Then there exist a decreasing function $\tau(t)$ defined in \eqref{equation:tau} and a constant $T$ defined in \eqref{T} such that the following hold:
\begin{enumerate}
	\item $
	  \sup\limits_{n\in \N}	\mathbb E \Big[\sup\limits_{s\in[0,T]} \|V_n\|_{\tau,r}^p +  \int_0^T \|A^r e^{\tau A} V_n\|^{p-2}\|A^{r+\frac12} e^{\tau A}V_n\|^2 ds \Big]  \leq C_p (1+ \mathbb E \|V_0\|_{\tau_0,r}^p) e^{C_p T};
	$
		\item For  $\alpha\in [0,1/2)$, $\int_0^{\cdot} e^{\tau A}\sigma_n(V_n) dW$ is bounded in
	$
		L^p\left( \Omega; W^{\alpha, p}(0, T; \mathcal D_{0,r-1}) \right);	
	$
	\item If moreover $p\geq 4$, then $e^{\tau A}V_n - \int_0^\cdot e^{\tau A}\sigma_n(V_n) dW$ is bounded in 
	$
		L^2\left( \Omega; W^{1, 2}(0, T; \mathcal D_{0,r-1}))\right) .
	$
\end{enumerate}
\end{lemma}
\begin{proof}
Applying the finite dimensional It\^{o} formula to \eqref{PE-inviscid-galerkin-1} produces, for $p\geq 2$, 
\begin{equation*}
    \begin{split}
        d\| V_n\|^p  = & - p \theta_\rho(\|V_n\|_{\tau,r} ) \|V_n\|^{p-2} \big\langle    Q_n(V_n, V_n) , V_n\big\rangle dt
         - p \|V_n\|^{p-2} \big\langle    F_n (V_n), V_n\big\rangle dt
        \\
        &+\frac p2 \| \sigma_n(V_n)\|^2_{L_2(\Uc, \mathcal D_0)} \|V_n\|^{p-2} dt
        + \frac{p(p-2)}2 \big\langle \sigma_n(V_n),  V_n \big\rangle^2 \|V_n\|^{p-4} dt
        \\
        &+ p \big\langle\sigma_n(V_n),  V_n \big\rangle \|V_n\|^{p-2} dW 
        \\
       := & A_1 dt +  A_2 dt +  A_3 dt +  A_4 dt +  A_5 dW.
    \end{split}
\end{equation*}
Thanks to the periodic boundary condition, integration by parts gives $ A_1=0$. By integration by parts, thanks to the Cauchy-Schwartz inequality, Young's inequality, and \eqref{force-inviscid},  one obtains that
\begin{equation*}
    \begin{split}
       \int_0^t A_2 ds &\leq \int_0^t  p \|V_n\|^{p-2} \Big|\big\langle    F_n (V_n), V_n\big\rangle\Big| ds
       = \int_0^t p \|V_n\|^{p-2} \Big|\big\langle    f, V_n\big\rangle\Big| ds
       \\
       &\leq \int_0^t p \|V_n\|^{p-2} (\|f\|^2 + \|V_n\|^2) ds \leq \frac14 \sup\limits_{s\in[0,t]} \|V_n\|^p + C_p\left(1 + \int_0^t \|V_n\|^p ds\right).
    \end{split}
\end{equation*}
From the assumptions \eqref{noise-inviscid} with $r = \tau = 0$, {\it i.e.}, $\|\sigma(V_n)\|^2_{L_2(\Uc, \mathcal D_0)} \leq C (1+ \|V_n\|^2 + \|A^{\frac12} V_n\|^2 )$, using the Cauchy-Schwartz inequality and Young's inequality yields that
\begin{equation*}
    \int_0^t | A_3|+| A_4| ds \leq C_p \int_0^t (1+ \|V_n\|^2 + \|A^{\frac12} V_n\|^2 ) \|V_n\|^{p-2} ds\leq C_p \int_0^t(1+ \|V_n\|_{\tau,r}^p ) ds.
\end{equation*}
Finally, the Burkholder-Davis-Gundy inequality, together with the Cauchy-Schwartz inequality, Young's inequality, and the property of $\sigma$ in \eqref{noise-inviscid} lead to 
\begin{equation*}
\begin{split}
    \mathbb E \sup\limits_{s\in[0,t]} \Big|\int_0^s A_5 dW \Big| &\leq C_p \mathbb E\Big(\int_0^t \|V_n\|^{2(p-1)} \| \|\sigma_n(V_n)\|^2_{L_2(\Uc, \mathcal D_0)} \Big)^{\frac12}
    \\
    &\leq C_p \mathbb E\Big(\int_0^t \|V_n\|^{2(p-1)}  (1+ \|V_n\|^2 + \|A^{\frac12} V_n\|^2 ) \Big)^{\frac12}
    \\
    &\leq \frac14 \mathbb E \sup\limits_{s\in[0,t]} \|V_n\|^p + C_p\mathbb E \int_0^t (1+ \|V_n\|_{\tau,r}^p) ds .
\end{split}
\end{equation*}

Next, for the seminorm $\|A^r e^{\tau A} V_n\|$, applying the finite dimensional It\^{o} formula yields, for $p\geq 2$, 
\begin{equation*}
    \begin{split}
        &d\|A^r e^{\tau A} V_n\|^p  - p\dot{\tau} \|A^r e^{\tau A}V_n\|^{p-2} \|A^{r+\frac12} e^{\tau A} V_n\|^2 dt
        \\
        = & - p \theta_\rho(\|V_n\|_{\tau,r} ) \|A^r e^{\tau A}V_n\|^{p-2} \big\langle A^r e^{\tau A}    Q_n(V_n, V_n) , A^r e^{\tau A} V_n\big\rangle dt
        \\
        & - p \|A^r e^{\tau A}V_n\|^{p-2} \big\langle A^r e^{\tau A}   F_n (V_n), A^r e^{\tau A}V_n\big\rangle dt
        \\
        &+\frac p2 \|A^r e^{\tau A} \sigma_n(V_n)\|^2_{L_2(\Uc, \mathcal D_0)} \|A^r e^{\tau A}V_n\|^{p-2} dt
        \\
        &+ \frac{p(p-2)}2 \big\langle A^r e^{\tau A} \sigma_n(V_n), A^r e^{\tau A} V_n \big\rangle^2 \|A^r e^{\tau A}V_n\|^{p-4} dt
        \\
        &+ p \big\langle A^r e^{\tau A} \sigma_n(V_n), A^r e^{\tau A} V_n \big\rangle \|A^r e^{\tau A}V_n\|^{p-2} dW 
        \\
       := &  B_1 dt +  B_2 dt +  B_3 dt +  B_4 dt +  B_5 dW.
    \end{split}
\end{equation*}
For the nonlinear terms $ B_1$, Lemma~\ref{lemma-inviscid} and properties of the cut-off function $\theta_\rho$ imply that   
\begin{equation*}
\begin{split}
    \int_0^t | B_1| &= p \theta_\rho(\|V_n\|_{\tau,r} ) \|A^r e^{\tau A}V_n\|^{p-2} \Big|\big\langle A^r e^{\tau A}    Q_n(V_n, V_n) , A^r e^{\tau A}V_n\big\rangle \Big|ds
    \\
    &\leq C_{r,p}  \int_0^t\theta_\rho(\|V_n\|_{\tau,r} ) \|V_n\|_{\tau,r} \|A^r e^{\tau A}V_n\|^{p-2}  \|A^{r+\frac12} e^{\tau A} V_n\|^2  ds
    \\
    &\leq C_{r,p}\int_0^t \rho\|A^r e^{\tau A}V_n\|^{p-2}  \|A^{r+\frac12} e^{\tau A} V_n\|^2 ds. 
\end{split}
\end{equation*}
Thanks to the periodic boundary condition, using integration by parts, the Cauchy-Schwartz inequality, Young's inequality, and \eqref{force-inviscid}, one obtains
\begin{equation*}
    \begin{split}
       \int_0^t B_2 ds &\leq  \int_0^t p \|A^r e^{\tau A}V_n\|^{p-2} \Big|\big\langle A^r e^{\tau A}   F_n (V_n), A^r e^{\tau A}V_n\big\rangle\Big| ds
       \\
       &= \int_0^t p \|A^r e^{\tau A}V_n\|^{p-2} \Big|\big\langle A^r e^{\tau A}   f, A^r e^{\tau A}V_n\big\rangle\Big| ds
       \\
       &\leq C_p \left(1  +  \int_0^t \|V_n\|_{\tau,r}^p ds \right) + \frac14 \sup\limits_{s\in[0,t]} \|A^r e^{\tau A} V_n\|^p.
    \end{split}
\end{equation*}
From \eqref{noise-inviscid}, we know that $\|A^r e^{\tau A} \sigma(V_n)\|^2_{L_2(\Uc, \mathcal D_0)} \leq C \Big(1+ \| V_n\|^2 + \|A^{r+\frac12} e^{\tau A}V_n\|^2 \Big)$, and thus
\begin{equation*}
\begin{split}
   \int_0^t | B_3|+| B_4| ds &\leq C_p \int_0^t (1+ \| V_n\|^2 + \|A^{r+\frac12} e^{\tau A}V_n\|^2 ) \|A^r e^{\tau A}V_n\|^{p-2} ds
    \\
    &\leq \int_0^t C_p (1+\|V_n\|_{\tau,r}^p + \|A^r e^{\tau A}V_n\|^{p-2} \|A^{r+\frac12} e^{\tau A}V_n\|^2)ds.
\end{split}
\end{equation*}
Finally, using the Burkholder-Davis-Gundy inequality, the Cauchy-Schwartz inequality, Young's inequality, and the property of $\sigma$ in \eqref{noise-inviscid}, we deduce that
\begin{equation*}
\begin{split}
    \mathbb E \sup\limits_{s\in[0,t]} \Big|\int_0^s B_5 dW \Big| \leq &C_p \mathbb E\Big(\int_0^t \|A^r e^{\tau A} V_n\|^{2(p-1)}  \|\sigma_n(V_n)\|^2_{L_2(\Uc, \mathcal D_{\tau,r})} \Big)^{\frac12}
    \\
    \leq &C_p \mathbb E\Big(\int_0^t \|A^r e^{\tau A} V_n\|^{2(p-1)}  (1+ \| V_n\|^2 + \|A^{r+\frac12} e^{\tau A}V_n\|^2 ) \Big)^{\frac12}
    \\
    \leq &\frac14 \mathbb E \sup\limits_{s\in[0,t]} \|A^r e^{\tau A} V_n\|^p 
    + C_p \mathbb E \int_0^t (1+ \|V_n\|_{\tau,r}^p + \|A^r e^{\tau A} V_n\|^{p-2}\|A^{r+\frac12} e^{\tau A}V_n\|^2 ) ds.
\end{split}
\end{equation*}

Combining the estimates of $A_1$ to $A_5$ and $B_1$ to $B_5$, thanks to \eqref{inviscid-norm-equi}, one obtains that
\begin{equation}\label{est:1}
    \begin{split}
        & \mathbb E \sup\limits_{s\in[0,t]} \|V_n\|_{\tau,r}^p +  \mathbb E \int_0^t \|A^r e^{\tau A} V_n\|^{p-2}\|A^{r+\frac12} e^{\tau A}V_n\|^2  ds
        \\
        \leq & \mathbb E \sup\limits_{s\in[0,t]} (2\|V_n\|^2 + 2\|A^r e^{\tau A} V_n\|^2)^{\frac p2} + \mathbb E \int_0^t \|A^r e^{\tau A} V_n\|^{p-2}\|A^{r+\frac12} e^{\tau A}V_n\|^2  ds
        \\
        \leq & C_p \Big( \mathbb E \sup\limits_{s\in[0,t]} \|V_n\|^p + \mathbb E \sup\limits_{s\in[0,t]} \|A^r e^{\tau A} V_n\|^p\Big) + \mathbb E \int_0^t \|A^r e^{\tau A} V_n\|^{p-2}\|A^{r+\frac12} e^{\tau A}V_n\|^2  ds
        \\
        \leq & C_p\big(\dot{\tau} + C_{r,p} (\rho+1)\big) \mathbb E \int_0^t \|A^r e^{\tau A} V_n\|^{p-2}\|A^{r+\frac12} e^{\tau A}V_n\|^2  ds
        \\
        &+ C_{p}\mathbb E \Big[ 1 +  \|(V_n)_0\|_{\tau_0,r}^p + \int_0^t \big(1 +  \|V_n\|_{\tau,r}^p   \big) ds \Big].
    \end{split}
\end{equation}
Here we have added the term $\mathbb E \int_0^t \|A^r e^{\tau A} V_n\|^{p-2}\|A^{r+\frac12} e^{\tau A}V_n\|^2  ds$ on both hand sides in order to get the corresponding regularity.
Let the deterministic function $\tau(t)$ satisfy
\begin{equation}\label{equation:tau}
    \tau(t) = \tau_0 - C_{r,p}(\rho + 1) t,
\end{equation}
and define the time 
\begin{equation}\label{T}
    T = \frac{\tau_0}{2C_{r,p}(\rho + 1)}.
\end{equation}
Observe that $\tau(t)>0$ and $\dot{\tau} + C_{r,p} (\rho+1) = 0$ for $t\in[0,T]$, and therefore \eqref{est:1} gives
\begin{equation*}
\begin{split}
    & \mathbb E \Big[\sup\limits_{s\in[0,t]} \|V_n\|_{\tau,r}^p +  \int_0^t \|A^r e^{\tau A} V_n\|^{p-2}\|A^{r+\frac12} e^{\tau A}V_n\|^2 ds \Big]
    \\
    \leq &C_{p}\mathbb E \Big[ 1+ \|V_0\|_{\tau_0,r}^p + \int_0^t \big(1 +  \|V_n\|_{\tau,r}^p   \big) ds \Big].
\end{split}
\end{equation*}
Thanks to the Gronwall inequality \cite[Lemma 5.3]{glatt2009strong}, for $p\geq 2$ and $t=T$, one has
\begin{equation}\label{eq:lemmaitem1}
    \mathbb E \Big[\sup\limits_{s\in[0,T]} \|V_n\|_{\tau,r}^p +  \int_0^T \|A^r e^{\tau A} V_n\|^{p-2}\|A^{r+\frac12} e^{\tau A}V_n\|^2 ds \Big]  \leq C_p(1+ \mathbb E \|V_0\|_{\tau_0,r}^p) e^{C_p T}.
\end{equation}

For part (ii), the fractional Burkholder-Davis-Gundy inequality \eqref{eq:bdg.frac}, \eqref{noise-inviscid}, and the estimate \eqref{eq:lemmaitem1} produce
\begin{equation*}
    \begin{split}
        \Eb \left| \int_0^\cdot e^{\tau A}\sigma_n\left(V_n\right) \, dW \right|^p_{W^{\alpha, p}\left(0, T; \mathcal D_{0,r-1}\right)} &\leq C_p \Eb \int_0^T \| e^{\tau A}\sigma_n\left(V_n\right) \|^p_{L_2\left(\Uc, \mathcal D_{0,r-1}\right)} \, ds 
        \\
        &\leq  C_p \Eb \left[  \int_0^T 1 + \|V_n\|_{\tau,r}^p \, ds \right] < \infty.
    \end{split}
\end{equation*}

Finally for (iii), from \eqref{PE-inviscid-galerkin-1} (or \eqref{PE-inviscid-galerkin-U-1}), for $t\in[0,T]$, one has
\begin{equation*}
    \begin{split}
        &e^{\tau A}V_n(t) - \int_0^t e^{\tau A}\sigma_n\left(V_n\right) \, dW 
        \\
        =& e^{\tau_0 A}V_n(0) - \int_0^t \theta_\rho(\|V_n\|_{\tau,r} ) e^{\tau A} Q_n(V_n, V_n) + e^{\tau A} F_n\left(V_n\right)  - \dot{\tau}A e^{\tau A} V_n \, ds.
    \end{split}
\end{equation*}
Then by Lemma~\ref{lemma-banach-algebra} and $r>\frac52$, the Minkowski inequality, and Young's inequality, from the properties of $\sigma$, $F$, and the cut-off function $\theta_\rho$, we have
\begin{equation}\label{estimate:w12}
\begin{split}
    & \Eb \left\| e^{\tau A}V_n - \int_0^\cdot e^{\tau A} \sigma_n\left(V_n\right) \, dW \right\|^2_{W^{1, 2}\left(0, T; \mathcal D_{0,r-1}\right)} 
    \\
    \leq &C_{T,f_0,p,\rho,r} \Eb \left[ 1+ \| V_0\|_{\tau_0,r}^2 + \sup\limits_{t\in[0,T]}\|V_n(t)\|_{\tau,r}^4+ \int_0^T   \| f\|_{\tau,r}^2 \, ds \right]
    < \infty,
\end{split}
\end{equation}
where the pressure gradient disappears since $\nabla P$ is orthogonal to the space $\mathcal D_{0,r-1}.$ Note that $p\geq 4$ is required due to applying the estimate \eqref{eq:lemmaitem1} to the term $\sup\limits_{t\in[0,T]}\|V_n(t)\|_{\tau,r}^4$ in \eqref{estimate:w12}.
\end{proof}
\begin{remark}\label{remark:glocal-modified}
The energy estimates in Lemma \ref{lemma:estimate-inviscid} hold up to a finite time $T$ defined in \eqref{T}. Due to this fact, the martingale solutions to the modified system \eqref{PE-inviscid-system-modified}, constructed below in Proposition~\ref{prop:global.martingale.existence}, can only exist on $[0,T]$ (see Proposition \ref{prop:approximating.sequence} and \ref{prop:global.martingale.existence}, below). Alternatively, one can replace the cut-off function $\theta_\rho$ by $\theta_{\rho \tau(t)}$ in the modified system \eqref{PE-inviscid-system-modified}, where $\tau(t)$ is the radius of analyticity, then energy estimates similar to Lemma \ref{lemma:estimate-inviscid} will hold for the new modified system with $\theta_{\rho \tau(t)}$ for any time $t\geq 0$. Moreover, based on this, the martingale solutions to this new modified system will exist globally in time. Nevertheless, the solution to the original one will still be local after applying a stopping time (see Corollary \ref{cor:loc.mart.sol}). 
\end{remark}

\begin{corollary}\label{cor:u}
With the same assumptions as in Lemma~\ref{lemma:estimate-inviscid}. Let $\tau(t)$ and $T$ be defined in \eqref{equation:tau} and \eqref{T}, respectively.
Suppose that $U_n = e^{\tau A}V_n$ is the solution to system \eqref{PE-inviscid-system-galerkin-U}, and we denote by $U_0 = e^{\tau_0A} V_0$. Then 
\begin{enumerate}
	\item $
	  \sup\limits_{n\in \N}	\mathbb E \Big[\sup\limits_{s\in[0,T]} \|U_n\|_{0,r}^p +  \int_0^T \|A^r  U_n\|^{p-2}\|A^{r+\frac12} U_n\|^2 ds \Big]  \leq (1+ \mathbb E \|U_0\|_{0,r}^p) e^{C_p T};
	$
		\item For $\alpha\in [0,1/2)$, $\int_0^{\cdot} e^{\tau A}\sigma_n(e^{-\tau A}U_n) dW$ is bounded in
	$
		L^p\bigl( \Omega; W^{\alpha, p}(0, T; \mathcal D_{0,r-1}) \bigr);	
	$
	\item If moreover $p\geq 4$, then $U_n - \int_0^\cdot e^{\tau A}\sigma_n(e^{-\tau A}U_n) dW$ is bounded in 
	$
		L^2\bigl( \Omega; W^{1, 2}(0, T; \mathcal D_{0,r-1})\bigr) .
	$
\end{enumerate}
\end{corollary}

\subsection{Local existence of martingale solutions}\label{section:existence-inviscid}

In this subsection, we establish the existence of the martingale solutions to system \eqref{PE-inviscid-system-abstract}. First, we need to construct the local martingale solutions to the modified system \eqref{PE-inviscid-system-modified}. For this purpose, we follow similar procedures as in \cite{brzezniak2021well,debussche2011local}. However, since our spaces $\mathcal D_{\tau(t),r}$ are changing with time, in order to use Aubin-Lions Lemma~\ref{lemma:aubin-lions}, one needs to work with $U_n$ instead of $V_n$ since $U_n$ lives in classical Sobolev spaces.

Given an initial distribution $\mu_0$ satisfying \eqref{eq:mu.zero} for some $p\geq 4$ (which is implied by \eqref{condition:mu-zero}), for some stochastic basis $\Sc = \left(\Omega, \Fc, \Fb, \Pb\right)$,
let $U_0= e^{\tau_0A}V_0$ be an $\Fc_0$-measurable $\mathcal D_{0,r}$-valued random variable with law $\mu_0$. For fixed $\rho \geq M$, let $U_n$ be the solutions to the approximating system \eqref{PE-inviscid-system-galerkin-U} with $U_n(0) = P_nU_0 $ on the stochastic basis $\Sc$. Define
\begin{equation*}
	\Xc_{U} = L^2\left(0, T; \mathcal D_{0,r} \right) \cap C\left(\left[0, T\right]; \mathcal D_{0,r-\frac32}\right), \quad \Xc_W = C\left(\left[0, T\right]; \Uc_0\right), \quad \Xc = \Xc_{U} \times \Xc_W,
\end{equation*}
where $T$ is defined in \eqref{T}. Notice that here $\mathcal D_{0,r}$ and $\mathcal D_{0,r-\frac32}$ are independent of time.
Let $\mu_{U}^n$, $\mu^n_W$ and $\mu^n$ be laws of $U_n$, $W$ and $(U_n, W)$ on $\Xc_{U}$, $\Xc_W$ and $\Xc$, respectively, in other words
\begin{equation}
	\label{eq:mu.measure}
	\mu^n_{U}(\cdot) = \Pb\left( \left\lbrace U_n \in \cdot \right\rbrace\right), \qquad \mu_W^n(\cdot) = \Pb\left( \left\lbrace W \in \cdot \right\rbrace \right), \qquad \mu^n = \mu_{U}^n \otimes \mu_W^n.
\end{equation}

The proof of the existence of the local martingale solutions to the modified system \eqref{PE-inviscid-system-modified} will be shown once we prove the following two propositions, which are following \cite[Proposition 4.1 and Proposition 7.1]{debussche2011local} or \cite[Proposition 3.2 and Proposition 3.3]{brzezniak2021well}. 

\begin{proposition}
\label{prop:approximating.sequence}
Let $\mu_0$ be a probability measure on $\mathcal D_{0,r}$ satisfying
\begin{equation}\label{condition:p}
    \int_{\mathcal D_{0,r}} \|U\|_{0,r}^p d\mu_0(U)<\infty  
\end{equation}
with $p\geq4$  
and let $(\mu^n)_{n \geq 1}$ be the measures defined in \eqref{eq:mu.measure}. Then there exists a probability space $(\tilde{\Omega}, \tilde{\Fc}, \tilde{\Pb})$, a subsequence $n_k \to \infty$ as $k \to \infty$ and a sequence of $\Xc$-valued random variables $(\tilde{U}_{n_k}, \tilde{W}_{n_k})$ such that
\begin{enumerate}
	\item $(\tilde{U}_{n_k}, \tilde{W}_{n_k})$ converges in $\Xc$ to $(\tilde{U}, \tilde{W}) \in \Xc$ almost surely,
	\item $\tilde{W}_{n_k}$ is a cylindrical Wiener process with reproducing kernel Hilbert space $\Uc$ adapted to the filtration $\left( \Fc_t^{n_k} \right)_{t \geq 0}$, where $\left( \Fc_t^{n_k} \right)_{t \geq 0}$ is the completion of $\sigma(\tilde{W}_{n_k}, \tilde{U}_{n_k}; 0 \leq s \leq t)$, 
	\item for $t\in[0,T]$, each pair $(\tilde{U}_{n_k}, \tilde{W}_{n_k})$ satisfies the equation 
	\begin{equation}
	\label{eq:appr.after.skorohod}
	\begin{split}
	    &d \tilde U_{n_k} + \big[\theta_\rho(\|\tilde U_{n_k}\|_{0,r} ) e^{\tau A} Q_{n_k}(e^{-\tau A}\tilde U_{n_k}, e^{-\tau A}\tilde U_{n_k}) + e^{\tau A}F_{n_k}(e^{-\tau A}\tilde U_{n_k}) -\dot{\tau} A \tilde U_{n_k}\big]dt  
	    \\
	    &= e^{\tau A}\sigma_{n_k} (e^{-\tau A}\tilde U_{n_k}) d\tilde W_{n_k} 
	\end{split}
	\end{equation}
	where $\tau$ and $T$ are defined in \eqref{equation:tau} and \eqref{T}, respectively. 
\end{enumerate}
\end{proposition}
\begin{remark}
In \cite[Proposition 3.2]{brzezniak2021well}, the authors required a condition similar to \eqref{condition:p} with $p\geq 8$. The reason is that the estimate (3.10) in \cite{brzezniak2021well} requires such a higher integrability condition. The estimation therein corresponds to \eqref{estimate:w12} in our derivation, which only requires $p\geq 4$. 
\end{remark}
\begin{proof}
For part (i), with Lemmas~\ref{lemma:aubin-lions} and \ref{lemma:estimate-inviscid}, one has the tightness of $\lbrace \mu^n \rbrace_{n\geq 1}$  in $\Xc$ following the argument in \cite[Lemma 4.1]{debussche2011local} with the spaces $D(A)$, $V$, $H$, and $V'$ therein being replaced by $\mathcal D_{0,r+\frac12}$, $\mathcal D_{0,r}$, $\mathcal D_{0,r-1}$, and $\mathcal D_{0,r-\frac32}$. Thus $\lbrace \mu^n \rbrace_{n\geq 1}$ is weakly compact by Prokhorov's theorem. Then, the first assertion follows immediately by the Skorokhod Theorem, see, {\it e.g.}, \cite[Theorem 2.4]{da2014stochastic}.

Parts (ii) and (iii) follow from the proof in \cite[Proposition 3.2]{brzezniak2021well}; see also \cite[Section 4.3.4]{bensoussan1995stochastic}.

\end{proof}

\begin{proposition}
\label{prop:global.martingale.existence}
For $\tau(t)$ and $T$ defined in \eqref{equation:tau} and \eqref{T}, respectively, let $(\tilde{U}_{n_k}, \tilde{W}_{n_k})$ be a sequence of $\Xc$-valued random variables on a probability space $(\tilde{\Omega}, \tilde{\Fc}, \tilde{\Pb})$ such that
\begin{enumerate}
	\item $(\tilde{U}_{n_k}, \tilde{W}_{n_k}) \to (\tilde{U}, \tilde{W})$ in the topology of $\Xc$ $\tilde{\Pb}$-almost surely, that is,
	\begin{equation*}
		\tilde{U}_{n_k} \to \tilde{U} \ \text{in} \ L^2\left(0, T; \mathcal D_{0,r}\right) \cap C\left(\left[0, T \right], \mathcal D_{0,r-\frac32}\right), \ \tilde{W}_{n_k} \to \tilde{W} \ \text{in} \ C\left(\left[0, T\right]; \Uc_0 \right),
	\end{equation*}
	\item $\tilde{W}_{n_k}$ is a cylindrical Wiener process with reproducing kernel Hilbert space $\Uc$ adapted to the filtration $\left( \Fc_t^{n_k} \right)_{t \geq 0}$ that contains  $\sigma(\tilde{W}_{n_k}, \tilde{U}_{n_k}; 0 \leq s \leq t)$,
	\item each pair $(\tilde{U}_{n_k}, \tilde{W}_{n_k})$ satisfies \eqref{eq:appr.after.skorohod}.
\end{enumerate}
Let 
\begin{equation}\label{u-to-v}
    \tilde V_{n_k} = e^{-\tau A} \tilde U_{n_k}, \ \ \tilde V = e^{-\tau A} \tilde U,
\end{equation}
and let $\tilde{\Fc}_t$ be the completion of $\sigma(\tilde{W}(s), \tilde{V}(s), 0 \leq s \leq t)$ and $\tilde{\Sc} = (\tilde{\Omega}, \tilde{\Fc}, ( \tilde{\Fc}_t )_{t \geq 0}, \tilde{\Pb})$. Then $(\tilde{\Sc}, \tilde{W}, \tilde{V},T)$ is a local martingale solution to the modified system \eqref{PE-inviscid-system-modified} on the time interval $[0,T]$. Moreover, $\tilde{V}$ satisfies
\begin{equation}
	\label{eq:martingale.solution.approx.regularity}
	\tilde{V} \in L^2\left( \tilde \Omega; C\left( [0, T]; \mathcal D_{\tau(t),r} \right) \right), \qquad  \|A^{r+\frac12} e^{\tau(t)A} \tilde{V} \|^2  \|A^{r} e^{\tau(t)A} \tilde{V}\|^{p-2} \in L^1\left(\tilde \Omega; L^1\left( 0, T \right) \right).
\end{equation}
\end{proposition}
\begin{proof}
First, due to the relation \eqref{u-to-v} and the assumpstion~(iii), one has
$$
d \tilde V_{n_k} + \big[\theta_\rho(\|\tilde V_n\|_{\tau,r} ) Q_{n_k}(\tilde V_{n_k}, \tilde V_{n_k}) + F_{n_k}(\tilde V_{n_k}) \big]dt   = \sigma_{n_k} (\tilde V_{n_k}) d\tilde W_{n_k}.
$$
Therefore, for any $\phi\in \mathcal D_0$ and $t\in[0,T]$, one has
\begin{equation}\label{weak-sol-galerkin}
\begin{split}
     \Big\langle \tilde V_{n_k}(t), \phi\Big\rangle + \int_0^t \Big\langle \theta_\rho(\| \tilde V_{n_k}\|_{\tau,r} ) Q_{n_k}(\tilde V_{n_k}, \tilde V_{n_k}) &+ F_{n_k}(\tilde V_{n_k}) ,\phi \Big\rangle ds 
	\\
	&= \Big\langle\tilde  V_{n_k}(0), \phi\Big\rangle + \int_0^t \Big\langle\sigma_{n_k}(\tilde V_{n_k}) , \phi \Big\rangle d\tilde{W}_{n_k}.
\end{split}
\end{equation}
From assumption~(i), we deduce that
\begin{equation}\label{convergence:v-in-analytic-r}
    \tilde V_{n_k} \to \tilde V \text{ in } L^2\left(0, T; \mathcal D_{\tau(t),r}\right) \cap C\left(\left[0, T \right], \mathcal D_{\tau(t),r-\frac32}\right) \; \tilde{\Pb}-a.s..
\end{equation}
Similarly as in \cite[Section 7.1]{debussche2011local}, by Corollary~\ref{cor:u} and \eqref{u-to-v}, one can establish that 
\begin{equation}\label{bound:v-and-vtilde}
\begin{split}
    &\tilde{V} \in L^2 \left( \tilde \Omega; L^2\big( 0,T; \mathcal D_{\tau(t),r+\frac12}   \big) \cap L^\infty\big( 0,T; \mathcal D_{\tau(t),r}   \big)\right),
    \\
    & \tilde{V}_{n_k} \rightharpoonup\tilde{V} \text{ in } L^2 \left( \tilde \Omega; L^2\big( 0,T; \mathcal D_{\tau(t),r+\frac12}   \big) \right),
    \\
    &  \tilde{V}_{n_k} \overset{\ast}{\rightharpoonup} \tilde{V} \text{ in } L^2 \left( \tilde \Omega; L^\infty\big( 0,T; \mathcal D_{\tau(t),r}   \big) \right).
\end{split}
\end{equation}
Moreover, from Lemma~\ref{lemma:estimate-inviscid} with $p >2$, one has the following uniform integrability for $\tilde {V}_{n_k}$: 
\begin{equation}\label{uniform-integrable}
\begin{split}
     \sup\limits_{k\in \mathbb N} \tilde \Eb \left[\left(\int_0^T \|\tilde V_{n_k}\|^2_{\tau,r} ds \right)^{\frac{p}{2}}  \right]
     \leq  C_T  \sup\limits_{k\in \mathbb N} \tilde \Eb \sup\limits_{t\in [0,T]} \|\tilde V_{n_k}\|_{\tau,r}^p   < \infty.
\end{split}
\end{equation}
The Vitali Convergence Theorem, together with \eqref{convergence:v-in-analytic-r} and \eqref{uniform-integrable}, implies that
\begin{equation}\label{convergence:vnk-probablity}
    \tilde{V}_{n_k} \rightarrow \tilde{V} \text{ in } L^2 \left( \tilde \Omega; L^2\big( 0,T; \mathcal D_{\tau(t),r}   \big) \right).
\end{equation}
Consequently, a further subsequence, still denoted by $\tilde V_{n_k}$ with a slightly abuse of notation, converges a.a. in $(0,T)\times \tilde \Omega$: 
\begin{equation}\label{convergence:Vnk-pw}
    \tilde{V}_{n_k} \rightarrow \tilde{V}  \text{ in } \mathcal D_{\tau(t),r} \text{ for a.a. } (t,\omega) \in (0,T)\times \tilde \Omega.
\end{equation}

\noindent \textbf{Convergence of the linear terms:}
The convergence of the initial condition is straightforward by \eqref{convergence:v-in-analytic-r}:
\begin{equation}\label{convergence:ini}
    \Big\langle\tilde  V_{n_k}(0), \phi\Big\rangle \rightarrow \Big\langle\tilde  V(0), \phi\Big\rangle = \Big\langle\tilde  V_0, \phi\Big\rangle.
\end{equation}
For any $t\in[0,T]$, thanks to the Cauchy-Schwarz inequality and \eqref{force-inviscid}, we know that
\begin{equation*}
\begin{split}
     &\Big| \int_0^t\Big\langle F_{n_k}(\tilde{V}_{n_k}) - F(\tilde{V}) , \phi  \Big\rangle ds  \Big| 
     \\
     = &\Big| \int_0^t\Big\langle f_0 (\tilde{V}_{n_k}^{\perp} - \tilde{V}^{\perp}) + (P_{n_k}\nabla P - \nabla P) + (P_{n_k}f - f) , \phi  \Big\rangle ds  \Big|
     \\
     \leq & f_0 \Big| \int_0^t\Big\langle  \tilde{V}_{n_k}^{\perp} - \tilde{V}^{\perp}, \phi  \Big\rangle ds  \Big| + \Big| \int_0^t\Big\langle  f , P_{n_k}\phi - \phi \Big\rangle ds  \Big| 
     \\
     \leq &C_{f_0} \|\phi\| \int_0^T \|\tilde{V}_{n_k}^{\perp} - \tilde{V}^{\perp}\|  \; ds + C \|P_{n_k}\phi - \phi\| \int_0^T \|f\| ds 
     \\
     \leq &C_{f_0,T} \|\phi\|  \left(\int_0^T \|\tilde{V}_{n_k}^{\perp} - \tilde{V}^{\perp}\|^2  \; ds \right)^{\frac12} + C_T \|P_{n_k}\phi - \phi\|. 
\end{split}
\end{equation*}
Therefore, by \eqref{convergence:Vnk-pw} and the property of $P_{n_k}$, for a.a. $(t,\omega) \in (0,T)\times \tilde \Omega$
\begin{equation}\label{convergence:force}
   \int_0^t\Big\langle F_{n_k}(\tilde{V}_{n_k})  , \phi  \Big\rangle ds \rightarrow \int_0^t\Big\langle F(\tilde{V}) , \phi  \Big\rangle ds .
\end{equation}

\noindent \textbf{Convergence of the nonlinear terms:}
For $t\in[0,T]$, one has
\begin{equation*}
    \begin{split}
        &\Big|\int_0^t \Big\langle \theta_\rho(\|\tilde V_{n_k}\|_{\tau,r} ) Q_{n_k}(\tilde{V}_{n_k}, \tilde{V}_{n_k}) - \theta_\rho(\| \tilde{V}\|_{\tau,r} ) Q(\tilde{V}, \tilde{V}) ,\phi \Big\rangle ds \Big|
        \\
        \leq &\Big|\int_0^t \Big\langle \theta_\rho(\|\tilde V_{n_k}\|_{\tau,r} ) \big( Q_{n_k}(\tilde{V}_{n_k}, \tilde{V}_{n_k}) - Q_{n_k}(\tilde{V}, \tilde{V}) \big)   ,\phi \Big\rangle ds \Big|
        \\
        &+ \Big|\int_0^t \Big\langle \theta_\rho(\|\tilde V_{n_k}\|_{\tau,r} )  Q(\tilde{V}, \tilde{V}) \big)   ,P_{n_k}\phi - \phi \Big\rangle ds \Big|
        \\
        &+ \Big|\int_0^t \Big\langle \big(\theta_\rho(\|\tilde V_{n_k}\|_{\tau,r} ) - \theta_\rho(\| \tilde{V}\|_{\tau,r} \big)  Q(\tilde{V}, \tilde{V}) \big)   , \phi \Big\rangle ds \Big| := I_1 + I_2 + I_3.
    \end{split}
\end{equation*}
For the term $I_1$, by the H\"older inequality and the Sobolev inequality, together with  \eqref{convergence:Vnk-pw}, one obtains that for a.a. $(t,\omega) \in (0,T)\times \tilde \Omega$, 
\begin{equation*}
\begin{split}
    I_1 &\leq C \|P_{n_k}\phi\| \int_0^T \|Q(\tilde{V}_{n_k}, \tilde{V}_{n_k}) - Q(\tilde{V}, \tilde{V}) \big)\| ds
    \\
    &\leq C \int_0^T \| Q(\tilde{V}_{n_k}- \tilde{V}, \tilde{V}_{n_k}) \|  +  \| Q( \tilde{V}, \tilde{V}_{n_k} - \tilde{V}) \| ds
    \\
    &\leq C \int_0^T \|\tilde{V}_{n_k}- \tilde{V}\|_{0,2} (\|\tilde{V}_{n_k}\|_{0,2} + \|\tilde{V}\|_{0,2} )
    \\
    &\leq C \left(\int_0^T \|\tilde{V}_{n_k}- \tilde{V}\|_{0,2}^2 \right)^{\frac12} \left(\int_0^T \|\tilde{V}_{n_k}\|^2_{0,2} + \|\tilde{V}\|^2_{0,2} \right)^{\frac12}
    \rightarrow 0.
\end{split}
\end{equation*}
For the term $I_2$, by the H\"older inequality and the Sobolev inequality, we have
\begin{equation*}
\begin{split}
    I_2 &\leq C \|\tilde{V}\|_{0,2}^2 \int_0^T \|P_{n_k}\phi - \phi\| ds \rightarrow 0,
\end{split}
\end{equation*}
for a.a. $(t,\omega) \in (0,T)\times \tilde \Omega$. For the term $I_3$, thanks to \eqref{convergence:Vnk-pw} and since $\theta_\rho$ is smooth, one has
\begin{equation}\label{convergence-theta}
    \theta_\rho(\|\tilde V_{n_k}\|_{\tau,r} ) \rightarrow \theta_\rho(\| \tilde{V}\|_{\tau,r} \big),
\end{equation}
for a.a. $(t,\omega) \in (0,T)\times \tilde \Omega$. Using \eqref{bound:v-and-vtilde} yields
\begin{equation*}
    \Eb \int_0^T \Big| \Big\langle \big(\theta_\rho(\|\tilde V_{n_k}\|_{\tau,r} ) - \theta_\rho(\| \tilde{V}\|_{\tau,r} \big)  Q(\tilde{V}, \tilde{V}) \big)   , \phi \Big\rangle  \Big|dt \leq C\Eb \int_0^T\|\tilde V\|_{0,2}^2 \|\phi\| dt <\infty.
\end{equation*}
The dominated convergence theorem together with \eqref{convergence-theta} yield that 
\begin{equation*}
    \begin{split}
        \Eb \int_0^T I_3 dt \leq C_T \Eb \int_0^T \Big| \Big\langle \big(\theta_\rho(\|\tilde V_{n_k}\|_{\tau,r} ) - \theta_\rho(\| \tilde{V}\|_{\tau,r} \big)  Q(\tilde{V}, \tilde{V}) \big)   , \phi \Big\rangle  \Big| dt \rightarrow 0.
    \end{split}
\end{equation*}
Thinning the sequence if necessary, we conclude that $I_3\rightarrow 0$ for a.a. $(t,\omega) \in (0,T)\times \tilde \Omega$. Combining the estimates of $I_1$ to $I_3$, for a.a. $(t,\omega) \in (0,T)\times \tilde \Omega$, 
\begin{equation}\label{convergence:deter}
    \int_0^t \Big\langle \theta_\rho(\|\tilde V_{n_k}\|_{\tau,r} ) Q_{n_k}(\tilde{V}_{n_k}, \tilde{V}_{n_k}) + F_{n_k}(\tilde{V}_{n_k})  ,\phi \Big\rangle ds \rightarrow
    \int_0^t \Big\langle  \theta_\rho(\| \tilde{V}\|_{\tau,r} ) Q(\tilde{V}, \tilde{V}) + F(\tilde{V}) ,\phi \Big\rangle ds.
\end{equation}

\noindent \textbf{Convergence of the stochastic terms:}
Using \eqref{noise-inviscid} and the Poincar\'e inequality \eqref{poincare}, we have
\begin{equation*}
    \begin{split}
        \| \sigma_{n_k}(\tilde V_{n_k}) - \sigma(\tilde V) &\|_{L^2\left(0, T; L_2\left(\Uc, \mathcal D_0\right)\right)}^2\\
	&\leq C \left( \| \sigma(\tilde V_{n_k}) - \sigma(\tilde V) \|_{L^2\left(0, T; L_2\left(\Uc, \mathcal D_0\right)\right)}^2 + \| (I-P_{n_k}) \sigma(\tilde V_{n_k}) \|_{L^2\left(0, T; L_2\left(\Uc, \mathcal D_0\right)\right)}^2 \right)\\
	&\leq C \left( \| \tilde V_{n_k} - \tilde V \|_{L^2(0, T; \mathcal D_{0,\frac12})}^2 + \frac{1}{n_k} \| \sigma(\tilde V_{n_k}) \|_{L^2(0, T; L_2(\Uc, \mathcal D_{0,\frac12}))}^2 \right)\\
	&\leq C \left( \| \tilde V_{n_k} - \tilde V \|_{L^2(0, T; \mathcal D_{0,\frac12})}^2 + \frac{1}{n_k} \int_0^T 1 + \| \tilde V_{n_k} \|_{0,1}^2  \, dt \right).
    \end{split}
\end{equation*}
Therefore, with Lemma~\ref{lemma:estimate-inviscid} and the convergence result \eqref{convergence:v-in-analytic-r} we have
\begin{equation*}
    \sigma_{n_k}(\tilde V_{n_k}) \to \sigma(\tilde V) \text{ in } L^2\left(0,T; L_2\left(\Uc, \mathcal D_0\right)\right), \; \tilde{\Pb}-a.s..
\end{equation*} 
In particular, this implies the convergence in probability in $L^2\left(0, T; L_2\left(\Uc, \mathcal D_0\right)\right).$ Thanks to \cite[Lemma 2.1]{debussche2011local}, from item (i), we obtain that
\begin{equation}
	\label{eq:sigma.dw.convergence}
	\int_0^\cdot \sigma_{n_k}(\tilde V_{n_k}) \, d\tilde{W}_{n_k} \to \int_0^\cdot \sigma(\tilde V) \, d\tilde{W},
\end{equation}
in probability in $L^2\left(0,T; \mathcal D_0\right)$. Thanks to Lemma~\ref{lemma:estimate-inviscid} and \eqref{noise-inviscid}, for $p>2$, we get
\begin{equation}
    	\label{eq:sigma.nk.lp}
    		\sup_{k \in \Nb} \Eb \left[ \int_0^T \| \sigma_{n_k}(\tilde V_{n_k}) \|_{L_2\left(\Uc, \mathcal D_0\right)}^{2} \, ds \right]^{\frac p2} \leq C_T \sup_{k \in \Nb} \Eb \left[ 1 + \sup_{t \in [0,T]} \| \tilde V_{n_k} \|_{0,\frac12}^{p} \right] <\infty.
\end{equation}
\eqref{eq:sigma.nk.lp} together with \eqref{eq:bdg} and the Vitali convergence theorem yield that the convergence \eqref{eq:sigma.dw.convergence} occurs in the space $L^2(\tilde{\Omega}; L^2\left(0,T; \mathcal D_0\right))$. Then for any $\mathcal R \subset \tilde \Omega \times [0,T]$ measurable, one has
\begin{equation*}
    \lim_{k \rightarrow \infty}\Eb 
     \int_0^T \chi_{\mathcal{R}} 
     \left( 
       \int_0^t \Big\langle 
      \sigma_{n_k}(\tilde{V}_{n_k})
       ,  \phi\Big\rangle  d\tilde W_{n_k}
    \right) dt=
  \Eb 
    \int_0^T \chi_{\mathcal{R}} 
    \left( 
      \int_0^t \Big\langle 
    \sigma(\tilde{V})
     ,  \phi\Big\rangle  d \tilde W
   \right) dt.
\end{equation*}
This implies that for a.a. $(t, \omega) \in [0,T] \times \tilde{\Omega}$,
\begin{equation}\label{convergence:sto}
   \int_0^t \Big\langle \sigma_{n_k}(\tilde V_{n_k}), \phi \Big\rangle d\tilde{W}_{n_k} \rightarrow \int_0^t \Big\langle \sigma(\tilde V), \phi \Big\rangle d\tilde{W}.
\end{equation}

Applying the convergences   \eqref{convergence:Vnk-pw}, \eqref{convergence:ini}, \eqref{convergence:force}, \eqref{convergence:deter} and \eqref{convergence:sto} to \eqref{weak-sol-galerkin}, we infer that for all $\phi \in \mathcal D_0$ and for a.a. $(t, \omega) \in [0,T] \times \tilde{\Omega}$,
\begin{equation*}
	\Big\langle\tilde V(t), \phi\Big\rangle + \int_0^t \Big\langle \theta_\rho(\| \tilde{V}\|_{\tau,r} ) Q(\tilde{V}, \tilde{V}) + F(\tilde{V}) ,\phi \Big\rangle ds 
	= \Big\langle\tilde V_0, \phi\Big\rangle + \int_0^t \Big\langle\sigma(\tilde V) , \phi \Big\rangle d\tilde{W}.
\end{equation*}
Therefore $(\tilde{\Sc}, \tilde{W}, \tilde{V}, T)$ is a local martingale solution to the modified system \eqref{PE-inviscid-system-modified}.

\noindent \textbf{Proof of the regularity \eqref{eq:martingale.solution.approx.regularity}:}
The proof of continuity of $\tilde V$ in time in the space $\mathcal D_{\tau(t),r}$ follows similarly as in \cite[Section 7.3]{debussche2011local}. For completeness, we highlight the main steps. By the property of $\sigma$ in \eqref{noise-inviscid} and the regularity of $\tilde V$ \eqref{bound:v-and-vtilde}, we have
\[
	\sigma(\tilde V) \in L^2\left(\tilde{\Omega}; L^2\left(0,T; L_2\left(\Uc, \mathcal D_{\tau(t),r}\right)\right)\right).
\]
Therefore, the solution to
\[
	dZ  = \sigma(\tilde V) \, d\tilde{W}, \qquad Z(0) = \tilde V_0,
\]
satisfies
\begin{equation}
	\label{eq:z.regularity}
	Z \in L^2\left(\tilde{\Omega}; C\left([0,T]; \mathcal D_{\tau(t),r}\right)\right) \cap L^2\left(\tilde{\Omega}; L^2\left(0,T; \mathcal D_{\tau(t),r+\frac12}\right)\right).
\end{equation}
Defining $\bar{V} = \tilde V - Z$, by \eqref{PE-inviscid-modified-1} we have $\Pb$-almost surely
\begin{equation*}
	\tfrac{d}{dt} \bar{V}  + \theta(\| \tilde V \|_{\tau,r}) Q(\tilde V, \tilde V) + F(\tilde V) = 0, \qquad \bar{V}(0) = 0.
\end{equation*}
Thanks \eqref{bound:v-and-vtilde} and Lemma~\ref{lemma-banach-algebra},
\[
	\ \theta(\| \tilde V  \|_{\tau,r}) Q(\tilde V, \tilde V) \text{ and } F(\tilde V) \in L^2\left(\tilde{\Omega}; L^2\left(0,T; \mathcal D_{\tau(t),r-\frac12}\right)\right)
\]
and thus
\[
	\frac{d}{dt} A^{r}e^{\tau(t)A} \bar{V} = A^{r}e^{\tau(t)A}  \frac{d}{dt} \bar{V} + \dot{\tau}A^{r+1}e^{\tau(t)A} \bar{V} \in L^2\left(\tilde{\Omega}; L^2\left(0,T; \mathcal D_{0,\frac12}'\right)\right),  
\]
\[
A^{r}e^{\tau(t)A} \bar{V} \in L^2\left(\tilde{\Omega}; L^2\left(0,T; \mathcal D_{0,\frac12}\right)\right),
\]
where $\mathcal D_{0,\frac12}'$ is the dual space of $\mathcal D_{0,\frac12}$. Since $\mathcal D_{0,\frac12} \subset \mathcal D_0 \equiv \mathcal D_0' \subset \mathcal D_{0,\frac12}'$, by the Lions-Magenes Lemma ({\it e.g.}, \cite[Lemma 1.2, Chapter 3]{temam2001navier}), we infer that $A^{r}e^{\tau(t)A} \bar{V} \in L^2\left(\tilde{\Omega}; C\left(\left[0,T \right]; \mathcal D_0\right)\right)$. Similarly, one can show that $\bar{V} \in L^2\left(\tilde{\Omega}; C\left(\left[0,T \right]; \mathcal D_0\right)\right)$, and thus $\bar{V} \in L^2\left( \Omega; C\left(\left[0,T \right]; \mathcal D_{\tau(t),r}\right)\right)$.
This together with \eqref{eq:z.regularity} imply that $\tilde{V} \in L^2\left( \Omega; C\left(\left[0,T \right]; \mathcal D_{\tau(t),r}\right)\right)$. The second part of \eqref{eq:martingale.solution.approx.regularity} follows directly from \eqref{bound:v-and-vtilde}.
\end{proof}

\begin{corollary}
\label{cor:loc.mart.sol}
Suppose that $\mu_0$ satisfy \eqref{condition:mu-zero} with constant $M>0$. Let $\rho \geq M$, and let $(\tilde{\Sc}, \tilde{W}, \tilde{V},T)$ be the local martingale solution to the modified system \eqref{PE-inviscid-system-modified} given in  Proposition~\ref{prop:global.martingale.existence}. Let
\begin{equation}\label{stopingtime:eta}
	\eta = \inf \left\lbrace t \geq 0 \mid \| \tilde{V}\|_{\tau,r} \geq \frac\rho2 \right\rbrace,
\end{equation}
Then $(\tilde{\Sc}, \tilde{W}, \tilde{V}, \eta\wedge T)$ is a local martingale solution to the problem \eqref{PE-inviscid-system-abstract}.
Moreover, 
\begin{equation*}
	\tilde{V}\left( \cdot \wedge \eta \right) \in L^2 \left(\tilde\Omega; C\left( [0, T], \mathcal D_{\tau(t),r} \right) \right), \quad \mathds{1}_{[0, \eta]} \|A^{r+\frac12} e^{\tau(t)A} \tilde{V} \|^2  \|A^{r} e^{\tau(t)A} \tilde{V}\|^{p-2} \in L^1\left(\tilde \Omega; L^1\left( 0, T \right) \right).
\end{equation*}
\end{corollary}

\subsection{Pathwise uniqueness}\label{section:uniqueness-inviscid}
In this section, we establish the pathwise uniqueness for the original system \eqref{PE-inviscid-system-abstract}. The following proposition is similar to \cite[Proposition 5.1]{debussche2011local} and \cite[Proposition 3.5]{brzezniak2021well}. However, in their results the uniqueness of martingales solutions to the corresponding modified equations is proven, while in our case, due to the different nonlinear estimates, only the uniqueness to the original system is obtained.

\begin{proposition}
\label{prop:pathwise.uniqueness}
Let $\tau(t)$ and $T$ defined as in \eqref{equation:tau} and \eqref{T}, respectively. Suppose that $\sigma$ and $f$ satisfy \eqref{noise-inviscid} and \eqref{force-inviscid}, respectively. 
Let $\Sc = \left(\Omega, \Fc, \Fct, \Pb \right)$ and $W$ be fixed.  
Suppose that there exist two local martingale solutions $\left(\Sc, W, V^1,T\right)$ and $\left(\Sc, W, V^2,T\right)$ to the modified system \eqref{PE-inviscid-system-modified}. Correspondingly, $\left(\Sc, W, V^1, \eta_1\wedge T\right)$ and $\left(\Sc, W, V^2, \eta_2\wedge T\right)$ are the two local martingale solutions to the original system \eqref{PE-inviscid-system-abstract}. Denote $\Omega_0 = \left\lbrace V^1(0) \equiv V^2(0) \right\rbrace \subseteq \Omega$, and $\eta = \eta_1 \wedge \eta_2$. Then
\begin{equation*}
	\Pb \left( \left\lbrace \mathds{1}_{\Omega_0}\left(V^1(t\wedge \eta) - V^2(t \wedge \eta)\right) = 0 \ \text{for all} \ t \in[0,T] \right\rbrace \right) = 1.
\end{equation*}
\end{proposition}

\begin{proof}
Let $R = V^1 - V^2$ and $\bar{R} = \mathds{1}_{\Omega_0} R$. Let $\eta^n$ be the stopping time defined by
\[
	\eta^n = \inf \left\lbrace t \geq 0 : \int_0^t \sum\limits_{i=1}^2\|V^i\|_{\tau,r} ^2 + \|A^{r+\frac12}e^{\tau A} V^i\|^2\, ds \geq n \right\rbrace,
\]
where the function $\tau(t)$ is defined in \eqref{equation:tau}.
Notice that both $\tau(t)$ and the existence time $T$ defined in \eqref{T} are independent of the initial data, thus both $V^1$ and $V^2$ are defined on the same interval $[0,T]$. Since $V_0 \in L^p\left( \Omega; \mathcal D_{\tau_0,r}\right)$ with $p \geq 4$, from the estimates in Lemma~\ref{lemma:estimate-inviscid} we deduce that $\lim\limits_{n\to \infty}\eta^n \geq T$ $\Pb$-a.s.\ and therefore it suffices to show that
\[
	\Eb \sup_{s \in \left[0, \eta \wedge \eta^n \wedge t \right]} \| \bar{R}(s) \|_{\tau,r}^2  = 0,
\]
for all $t \in (0,T]$ and $n \in \Nb$ such that $\eta^n \leq T$. Subtracting the equation of $V^2$ from the one of $V^1$ one has
\begin{equation*}
    \begin{split}
       &dR + \left[ Q(V^1,V^1) -  Q(V^2,V^2) + f_0  R^\perp + \nabla (P^1-P^2)\right] \, dt
        = \left[ \sigma(V^1) - \sigma(V^2) \right] \, dW, 
	   \\
	   &R(0) = V^1(0) - V^2(0).
    \end{split}
\end{equation*}
Fix $n \in \Nb$ and let $\eta_a$, $\eta_b$ be stopping times such that $0 \leq \eta_a \leq \eta_b \leq \eta\wedge\eta^n \wedge T$. We now calculate
\begin{align}
	\nonumber
	\Eb  &\sup_{s \in \left[\eta_a, \eta_b\right]} \| \bar{R} \|^2  \leq  \Eb \| \bar{R}(\eta_a) \|^2 \\
	\nonumber
	&\quad+ 2 \Eb \int_{\eta_a}^{\eta_b} \mathds{1}_{\Omega_0}\left| \Big\langle Q(V^1,V^1) - Q(V^2,V^2), \bar{R}\Big\rangle \, ds \right| + 2 \Eb \int_{\eta_a}^{\eta_b} \left| \Big\langle f_0 \bar R^\perp + \mathds{1}_{\Omega_0}\nabla (P^1-P^2), \bar{R}\Big\rangle \right| \, ds \\
	\nonumber
	&\quad+ 2 \Eb \sup_{s \in \left[\eta_a, \eta_b\right]} \left| \int_{\eta_a}^{s} \mathds{1}_{\Omega_0}\Big\langle  \sigma(V^1) - \sigma(V^2)  , \bar{R} \Big\rangle dW \right|+ \Eb \int_{\eta_a}^{\eta_b} \mathds{1}_{\Omega_0}\| \sigma(V^1) - \sigma(V^2) \|_{L_2\left(\Uc, \mathcal D_0\right)}^2 \, ds \\
	\label{eq:R.estimate}
	&= \Eb \| \bar{R}(\eta_a) \|^2 + A_1 + A_2 + A_3 + A_4 .
\end{align}
where the It\^{o}'s Lemma, integrating in time and taking supremums, multiplying by $\mathds{1}_{\Omega_0}$ and taking the expected value have been successively applied. 

For $A_1$, thanks to the the H\"older inequality and the Sobolev inequality:
\begin{align*}
	A_1 &\leq 2 \Eb \int_{\eta_a}^{\eta_b} \left| \Big\langle Q(\bar R, V^1) +  Q(V^2, \bar R) , \bar{R} \Big\rangle \right| \, ds
	\leq C \Eb\int_{\eta_a}^{\eta_b} \left(\|V^1\|_{0,2} + \|V^2\|_{0,2} \right)\|\bar R\|_{0,2}^2 ds.
\end{align*}
Since $\langle \bar R^\perp, \bar R\rangle = 0$ and $\langle\mathds{1}_{\Omega_0} \nabla (P^1-P^2), \bar R\rangle = 0$, one has $A_2= 0$.
Regarding $A_3$, the Burkholder-Davis-Gundy inequality \eqref{eq:bdg}, the Lipschitz continuity of in $\sigma$ \eqref{noise-inviscid}, and the Young inequality give
\begin{equation*}
	\begin{split}
		A_3 \leq C \Eb \left( \int_{\eta_a}^{\eta_b} \mathds{1}_{\Omega_0}\| \sigma(V^1) - \sigma(V^2) \|_{L_2\left(\Uc, \mathcal D_0\right)}^2 \| \bar{R} \|^2 \, ds \right)^{1/2}\leq \frac1{2} \Eb \sup_{s \in \left[\eta_a, \eta_b\right]} \| \bar{R} \|^2 + C \Eb \int_{\eta_a}^{\eta_b} \| \bar{R} \|_{0,\frac12}^2 ds.
	\end{split}
\end{equation*}
Finally, the integral $A_4$ is estimated using the Lipschitz continuity of $\sigma$ in \eqref{noise-inviscid}. We get
\begin{equation*}
	A_4 \leq  C \Eb \int_{\eta_a}^{\eta_b} \| \bar{R} \|_{0,\frac12}^2 \, ds.
\end{equation*}

For the higher order part, thanks to \eqref{bound:v-and-vtilde}, one is allowed to take inner product of $dA^{r}e^{\tau A} R$ with $A^{r}e^{\tau A} R$ and follow a similar derivation of \eqref{eq:R.estimate} to get
\begin{align*}
	\nonumber
	\Eb  \sup_{s \in \left[\eta_a, \eta_b\right]} &\| A^r e^{\tau A} \bar{R} \|^2  \leq  \Eb \|A^r e^{\tau(\eta_a) A} \bar{R}(\eta_a) \|^2 + 2 \Eb \int_{\eta_a}^{\eta_b} \dot\tau \| A^{r+\frac12} e^{\tau A} \bar{R} \|^2 ds  \\
	\nonumber
	&\quad+ 2 \Eb \int_{\eta_a}^{\eta_b} \mathds{1}_{\Omega_0}\left| \Big\langle A^r \Big(e^{\tau A} Q(V^1,V^1) - Q(V^2,V^2)\Big), A^r e^{\tau A}\bar{R}\Big\rangle \, ds \right| \\ \nonumber
	&\quad + 2 \Eb \int_{\eta_a}^{\eta_b} \left| \Big\langle f_0 A^r e^{\tau A} \bar R^\perp + \mathds{1}_{\Omega_0} A^r e^{\tau A} \nabla (P^1-P^2), A^r e^{\tau A} \bar{R}\Big\rangle \right| \, ds \\
	\nonumber
	&\quad+ 2 \Eb \sup_{s \in \left[\eta_a, \eta_b\right]} \left| \int_{\eta_a}^{s} \mathds{1}_{\Omega_0}\Big\langle  A^r e^{\tau A} \Big(\sigma(V^1) - \sigma(V^2) \Big) , A^r e^{\tau A}\bar{R} \Big\rangle dW \right| 
	\\ \nonumber
	&\quad+ \Eb \int_{\eta_a}^{\eta_b} \mathds{1}_{\Omega_0}\| \sigma(V^1) - \sigma(V^2) \|_{L_2\left(\Uc, \mathcal D_{\tau, r}\right)}^2 \, ds \\
	&= \Eb \| \bar{R}(\eta_a) \|^2 + 2 \Eb \int_{\eta_a}^{\eta_b} \dot\tau \| A^{r+\frac12} e^{\tau A} \bar{R} \|^2 ds + B_1 + B_2 + B_3 + B_4.
\end{align*}
Similar to $A_1-A_4$,
by applying Lemma~\ref{lemma-inviscid} and since $\eta_b \leq \eta$, one obtains that
\begin{equation}\label{nonlinear-trouble}
    \begin{split}
        B_1 &\leq 2 \Eb \int_{\eta_a}^{\eta_b} \left| \Big\langle A^r e^{\tau A} \Big(\ Q(\bar R, V^1) +  Q(V^2, \bar R) \Big) ,A^r e^{\tau A} \bar{R} \Big\rangle \right| \, ds
	    \\
	    &\leq C_r \Eb \int_{\eta_a}^{\eta_b} \rho \| A^{r+\frac12} e^{\tau A} \bar{R} \|^2 ds + C_r \Eb \int_{\eta_a}^{\eta_b}  \sum\limits_{i=1}^2 \| A^{r+\frac12} e^{\tau A} V^i \| \|\bar{R} \|_{\tau,r} \| A^{r+\frac12} e^{\tau A} \bar{R} \| ds
	    \\
	    &\leq C_r\Eb \int_{\eta_a}^{\eta_b} (\rho + 1) \| A^{r+\frac12} e^{\tau A} \bar{R} \|^2  ds + C_r\Eb \int_{\eta_a}^{\eta_b} \sum\limits_{i=1}^2\| A^{r+\frac12} e^{\tau A} V^i \|^2 \|\bar{R} \|_{\tau,r}^2  ds.
    \end{split}
\end{equation}
Since $\langle f_0 A^r e^{\tau A} \bar R^\perp , A^r e^{\tau A} \bar{R}\rangle = 0$ and $ \langle  \mathds{1}_{\Omega_0} A^r e^{\tau A} \nabla (P^1-P^2), A^r e^{\tau A} \bar{R}\rangle = 0$, one has $B_2= 0$.
Regarding $B_3$, the Burkholder-Davis-Gundy inequality \eqref{eq:bdg}, the Lipschitz continuity of $\sigma$ in  \eqref{noise-inviscid}, and the Young inequality give
\begin{equation*}
	\begin{split}
		B_3 \leq &C \Eb \left( \int_{\eta_a}^{\eta_b} \mathds{1}_{\Omega_0}\| \sigma(V^1) - \sigma(V^2) \|_{L_2\left(\Uc, \mathcal D_{\tau,r}\right)}^2\| A^{r} e^{\tau A} \bar{R} \|^2 \, ds \right)^{1/2}
		\\
		\leq &\frac1{2} \Eb \sup_{s \in \left[\eta_a, \eta_b\right]} \| A^{r} e^{\tau A} \bar{R} \|^2 + C \Eb \int_{\eta_a}^{\eta_b} \|\bar{R} \|_{\tau,r}^2 + \| A^{r+\frac12} e^{\tau A} \bar{R} \|^2 ds.
	\end{split}
\end{equation*}
Finally, the integral $B_4$ is estimated using the Lipschitz continuity of  $\sigma$ in \eqref{noise-inviscid}. Therefore,
\begin{equation*}
	B_4 \leq  C \Eb \int_{\eta_a}^{\eta_b} \|\bar{R} \|_{\tau,r}^2 +\| A^{r+\frac12} e^{\tau A} \bar{R} \|^2 ds.
\end{equation*}

Collecting the estimates of $A_1$ to $A_4$ and $B_1$ to $B_4$ gives
\begin{equation}\label{est:uniqueness}
    \begin{split}
        \Eb  \sup_{s \in \left[\eta_a, \eta_b\right]} \|  \bar{R}(s) \|_{\tau(s),r}^2 \leq & C\Eb \|  \bar{R} (\eta_a) \|_{\tau(\eta_a),r}^2 + C \Eb \int_{\eta_a}^{\eta_b} \left(\dot\tau + C_r\left(\rho+1 \right)\right) \| A^{r+\frac12} e^{\tau A} \bar{R} \|^2 \, ds 
        \\
        & + C_r\Eb \int_{\eta_a}^{\eta_b} \left(1 +  \sum\limits_{i=1}^2 \left(\|  V^i \|_{\tau,r}^2  +  \| A^{r+\frac12} e^{\tau A} V^i \|^2 \right) \right)  \|  \bar{R} \|_{\tau,r}^2 \, ds.
    \end{split}
\end{equation}
By taking the constant $C_{r,\rho}$ appearing in the definition of $\tau$ in \eqref{equation:tau} large enough so that $C_{r,\rho} \geq C_r$ where $C_r$ appearing in \eqref{est:uniqueness}, one obtains that $\dot\tau + C_r\left(\rho+1 \right) \leq 0$. With this and \eqref{est:uniqueness} one can apply the stochastic Gronwall Lemma (see \cite[Lemma 5.3]{glatt2009strong}) to conclude the proof.
\end{proof}

\begin{remark}\label{remark:nonlinear-trouble}
Unlike \cite{brzezniak2021well,debussche2011local,saal2021stochastic} where the authors were able to show the strong uniqueness of the martingale solutions to the modified system, we are only able to show it for the original system. The reason is that we need the help of stopping times to control the terms $\| A^{r} e^{\tau A} V^i \|$ in \eqref{nonlinear-trouble}. With stopping times, we are considering the solutions to the original system instead of the modified system. The essential reason behind this is that the horizontal viscosity can help control the nonlinear estimates in \cite{brzezniak2021well,debussche2011local,saal2021stochastic}, while the dissipation from analytic framework is not strong enough to provide the same control in our case.
\end{remark}

%%%%%%%%%%%%%%%%%%%%%%%%%
%%%%%%%%%%%%%%%%%%%%%%%%%

\section{The Vertically Viscous Case}\label{section:viscous}
In this section, we study system \eqref{PE-viscous-system-abstract}, {\it i.e.}, the stochastic PEs with vertical viscosity. For some fixed $\rho\geq M$ with $M$ appearing in the condition \eqref{condition:mu-zero-viscous}, and the cut-off function $\theta_\rho$ defined in \eqref{eqn:rho},  
we consider the following modified system
\begin{subequations}\label{PE-viscous-system-modified}
\begin{align}
    &d V + \left[\theta_\rho\left(\| V\|_{\tau,r,\gamma,r}  \right)Q( V, V) + F(V) - \nu_z \partial_{zz} V\right]dt = \sigma (V) dW , \label{PE-viscous-modified-1}  
    \\
    &V(0) = V_0, \label{PE-viscous-modified-ic}
\end{align}
\end{subequations}
where the functions $\tau(t)$ and $\gamma(t)$ in $\| V\|_{\tau,r,\gamma,r} $ will be determined later.
Following the same setup as in Section~\ref{section:galerkin}, for each $n\in \N$, we consider the following Galerkin approximation of system \eqref{PE-viscous-system-modified} at order $n$ as:
\begin{subequations}\label{PE-viscous-system-galerkin}
\begin{align}
    &d V_n + \left[\theta_\rho\left(\| V_n\|_{\tau,r,\gamma,r} \right) Q_n(V_n, V_n) + F_n(V_n) - \nu_z \partial_{zz} V_n \right]dt   = \sigma_n (V_n) dW , \label{PE-viscous-galerkin-1}  
    \\
    & V_n(0) = P_n V_0 := (V_n)_0. \label{PE-viscous-galerkin-ic}  
\end{align}
\end{subequations}
Define $U_n = e^{\tau A_h}e^{\gamma A_z}V_n$, and therefore, $V_n = e^{-\tau A_h}e^{-\gamma A_z}U_n$. Since $dU_n = (\dot{\tau}A_h U_n  + \dot\gamma A_z U_n)dt + e^{\tau A_h}e^{\gamma A_z}dV_n$, system \eqref{PE-viscous-system-galerkin} is equivalent to
\begin{subequations}\label{PE-viscous-system-galerkin-U}
\begin{align}
    &d U_n + \Big[\theta_\rho\left(\|U_n\|_{0,r}  \right) e^{\tau A_h} e^{\gamma A_z} Q_n(e^{-\tau A_h} e^{-\gamma A_z}U_n, e^{-\tau A_h} e^{-\gamma A_z}U_n) \nonumber
    \\
    &\qquad+ e^{\tau A_h} e^{\gamma A_z}F_n(e^{-\tau A_h} e^{-\gamma A_z}U_n) -\dot{\tau}A_h U_n  - \dot\gamma A_z U_n - \nu_z \partial_{zz} U_n\Big]dt   
    \\
    &= e^{\tau A_h} e^{\gamma A_z}\sigma_n (e^{-\tau A_h} e^{-\gamma A_z}U_n) dW , \label{PE-viscous-galerkin-U-1}  
    \\
    & U_n(0) = e^{\tau_0 A_h }V_n(0), \label{PE-viscous-galerkin-U-ic}  
\end{align}
\end{subequations}
and $U_n = e^{\tau A_h}e^{\gamma A_z}V_n$ is the solution to system \eqref{PE-viscous-system-galerkin-U}. Notice that initially the analytic radius in $z$ is $\gamma_0 =0$, thus $e^{\gamma_0 A_z} = 1$.

This section aims to provide results that are parallel to the inviscid case summarized in Theorem~\ref{theorem:viscous}. For this purpose, we fix $\rho>M$, $\tau_0>0$, $\gamma_0=0$, and $r>\frac52$ through this section. In Section~\ref{section:energy-estimate-viscous}, we establish the energy estimates of the Galerkin systems \eqref{PE-viscous-system-galerkin} and \eqref{PE-viscous-system-galerkin-U}. Section~\ref{section:existence-uniqueness-viscous} will focus on the local  existence martingale solutions and pathwise uniqueness to the original system \eqref{PE-viscous-system-abstract}, where we only highlight some necessary details as most of the proofs are similar to the inviscid case given in Sections~\ref{section:existence-inviscid} and \ref{section:uniqueness-inviscid}. This will complete the proof of Theorem~\ref{theorem:viscous}.

\subsection{Energy estimates} \label{section:energy-estimate-viscous}
We establish the main estimates needed to pass to the limit in the Galerkin system \eqref{PE-viscous-system-galerkin} (and equivalently, system \eqref{PE-viscous-system-galerkin-U}). Recall that when $\tau = \gamma =0$ and $r=s$, $\mathcal D_{0,r,0,s} = \mathcal D_{0,r}$.

\begin{lemma}\label{lemma:estimate-viscous}
Let $\tau_0>0, \gamma_0 = 0$, $p\geq 2$, and $r>\frac52$ be fixed. Suppose that $V_0 \in L^p(\Omega; \mathcal D_{\tau_0,r,0,r})$ be an $\Fc_0$-measurable random variable. Assume that $\sigma$ satisfies \eqref{noise-viscous} with $\delta$ satisfying \eqref{assumption:delta}, and let $f$ satisfies \eqref{force-viscous}.
Let $V_n$ be the solutions of system \eqref{PE-viscous-system-galerkin}. Then there exist functions $\tau(t)$ and $\gamma(t)$ defined in \eqref{tau-gamma}, and a constant $T$ defined in \eqref{T-viscous} such that the following hold:

 \begin{multline*}
\text{(i)} \quad \mathbb E \Big[\sup\limits_{t\in[0,T]} \|V_n\|_{\tau,r,\gamma,r}^p +  \int_0^T \nu_z \|A^r e^{\tau A_h}  e^{\gamma A_z}V_n\|^{p-2} \|\partial_z V_n\|_{\tau,r,\gamma,r}^2 ds
    \\
    + \int_0^T \|A^r e^{\tau A_h}  e^{\gamma A_z}V_n\|^{p-2} \|A_h^{\frac12} A^r e^{\tau A_h}  e^{\gamma A_z} V_n\|^2 ds \Big]  \leq C_p(1+ \mathbb E \|V_0\|_{\tau_0,r,0,r}^p) e^{C_p T}.
\end{multline*}
	
		(ii) Let  $\alpha\in [0,1/2)$. Then $\int_0^{\cdot} e^{\tau A_h}  e^{\gamma A_z}\sigma_n(V_n) dW$ is bounded in
	$
		L^p\left( \Omega; W^{\alpha, p}(0, T; \mathcal D_{0,r-1}) \right);	
	$

	(iii) If moreover $p\geq 4$, $e^{\tau A_h}  e^{\gamma A_z}V_n - \int_0^\cdot e^{\tau A_h}  e^{\gamma A_z}\sigma_n(V_n) dW$ is bounded in 
	$
		L^2\left( \Omega; W^{1, 2}(0, T; \mathcal D_{0,r-1}))\right) .
	$
\end{lemma}

\begin{remark}\label{remark:radius-analyticity}
Different from the inviscid case where the radius of analyticity is shrinking on all variables, from \eqref{tau-gamma} below, one will see that the radius of analyticity in the $z$ variable starts from $0$ and is increasing thanks to the vertically viscousity.
\end{remark}

\begin{proof}
For $\|V_n\|$, by applying the finite dimensional It\^{o} formula to \eqref{PE-viscous-galerkin-1}, for $p\geq 2$, one has
\begin{equation*}
    \begin{split}
        &d\| V_n\|^p + p\nu_z\|A_z  V_n\|^2 \| V_n\|^{p-2}  
        \\
        = & - p \theta_\rho\left(\| V_n\|_{\tau,r,\gamma,r}  \right) \|V_n\|^{p-2} \big\langle    Q_n(V_n, V_n) ,  V_n\big\rangle dt
        - p \|V_n\|^{p-2} \big\langle    F_n (V_n), V_n\big\rangle dt
        \\
        &+\frac p2 \| \sigma_n(V_n)\|^2_{L_2(\Uc, \mathcal D_0)} \|V_n\|^{p-2} dt
        + \frac{p(p-2)}2 \big\langle  \sigma_n(V_n),  V_n \big\rangle^2 \|V_n\|^{p-4} dt
        \\
        &+ p \big\langle  \sigma_n(V_n),  V_n \big\rangle \| V_n\|^{p-2} dW 
        \\
       := & A_1 dt +  A_2 dt +  A_3 dt +  A_4 dt +  A_5 dW.
    \end{split}
\end{equation*}
For $\|A^r e^{\tau A_h}  e^{\gamma A_z}V_n\|$, applying the finite dimensional It\^{o} formula gives, for $p\geq 2$,
\begin{equation*}
    \begin{split}
        &d\|A^r e^{\tau A_h}  e^{\gamma A_z}V_n\|^p + p\nu_z \|A^r e^{\tau A_h}  e^{\gamma A_z}V_n\|^{p-2} \|A^r A_z e^{\tau A_h}  e^{\gamma A_z}V_n\|^2
        \\
        & - p\big( \dot{\tau} \|A^r e^{\tau A_h}  e^{\gamma A_z}V_n\|^{p-2} \|A_h^{\frac12} A^r  e^{\tau A_h}  e^{\gamma A_z}V_n\|^2 
        \\
        &\hspace{3cm}+ \dot{\gamma} \|A^r e^{\tau A_h}  e^{\gamma A_z}V_n\|^{p-2} \| A_z^{\frac12} A^{r}  e^{\tau A_h}  e^{\gamma A_z} V_n\|^2 \big) dt
        \\
        = & - p \theta_\rho\left(\| V_n\|_{\tau,r,\gamma,r}  \right) \|A^r e^{\tau A_h}  e^{\gamma A_z}V_n\|^{p-2} \big\langle A^r e^{\tau A_h}  e^{\gamma A_z}    Q_n(V_n, V_n) , A^r e^{\tau A_h}  e^{\gamma A_z} V_n\big\rangle dt
        \\
        & - p \|A^r e^{\tau A_h}  e^{\gamma A_z}V_n\|^{p-2} \big\langle A^r e^{\tau A_h}  e^{\gamma A_z}   F_n (V_n), A^r e^{\tau A_h}  e^{\gamma A_z}V_n\big\rangle dt
        \\
        &+\frac p2 \| A^r e^{\tau A_h}  e^{\gamma A_z}\sigma_n(V_n)\|^2_{L_2(\Uc, \mathcal D_0)} \|A^r e^{\tau A_h}  e^{\gamma A_z}V_n\|^{p-2} dt
        \\
        &+ \frac{p(p-2)}2 \big\langle A^r e^{\tau A_h}  e^{\gamma A_z} \sigma_n(V_n), A^r e^{\tau A_h}  e^{\gamma A_z} V_n \big\rangle^2 \|A^r e^{\tau A_h}  e^{\gamma A_z}V_n\|^{p-4} dt
        \\
        &+ p \big\langle A^r e^{\tau A_h}  e^{\gamma A_z} \sigma_n(V_n), A^r e^{\tau A_h}  e^{\gamma A_z} V_n \big\rangle \|A^r e^{\tau A_h}  e^{\gamma A_z}V_n\|^{p-2} dW 
        \\
       := &  B_1 dt +  B_2 dt +  B_3 dt +  B_4 dt +  B_5 dW.
    \end{split}
\end{equation*}
By integration by parts and due to the boundary condition and incompressible condition, $A_1 = 0$. For the nonlinear term $B_1$,  Lemma~\ref{lemma-viscous-type1}, Young's inequality and the property of cut-off function $\theta_\rho$ together give 
\begin{equation*}
    \begin{split}
        |B_1| \leq& C_{p,r} \theta_\rho\left(\| V_n\|_{\tau,r,\gamma,r} \right) \|A^r e^{\tau A_h}  e^{\gamma A_z}V_n\|^{p-2} 
        \\
        & \times \left(\|V_n\|_{\tau,r,\gamma,r} \|A_h^{\frac12} A^r e^{\tau A_h}  e^{\gamma A_z} V\|^2  + \|V_n\|_{\tau,r,\gamma,r} \|\partial_z V_n\|_{\tau,r,\gamma,r} \|A_h^{\frac12} A^r e^{\tau A_h}  e^{\gamma A_z} V\|\right)
        \\
        \leq & C_{p,r,\nu_z} (1+\rho^2) \|A^r e^{\tau A_h}  e^{\gamma A_z}V_n\|^{p-2} \|A_h^{\frac12} A^r e^{\tau A_h}  e^{\gamma A_z} V\|^2 + \frac{p\nu_z}8 \|A^r e^{\tau A_h}  e^{\gamma A_z}V_n\|^{p-2} \|\partial_z V_n\|_{\tau,r,\gamma,r}^2 .
    \end{split}
\end{equation*}
For $A_2$ and $B_2$, by integration by parts, thanks to the periodic boundary condition, the Cauchy-Schwartz inequality, and Young's inequality, and noticing that $\gamma(T) = \frac{\nu_z}8 T \leq \frac{\nu_z}8 \tau_0 \leq \gamma^*$ from \eqref{force-viscous}, \eqref{tau-gamma}, and \eqref{T-viscous},
one obtains that
\begin{equation*}
    \begin{split}
       \int_0^t |A_2| + |B_2| ds &\leq C_p\int_0^t  \|V_n\|_{\tau,r,\gamma,r}^{p-2} \Big( \Big|\big\langle A^r e^{\tau A_h}  e^{\gamma A_z}   f, A^r e^{\tau A_h}  e^{\gamma A_z}V_n\big\rangle \Big| + \Big|\big\langle  f,  V_n\big\rangle \Big| \Big)ds
       \\
       &\leq \int_0^t C_p \|V_n\|_{\tau,r,\gamma,r}^{p-2} \left(\|f\|_{\tau_0,r,\gamma(T),r}^2 + \|V_n\|_{\tau,r,\gamma,r}^2 \right) ds 
       \\
       &\leq \frac14 \left(\sup\limits_{s\in[0,t]} \|V_n\|^p + \sup\limits_{s\in[0,t]}\|A^r e^{\tau A_h}  e^{\gamma A_z}  V_n\|^p  \right) + C_p\left(1 + \int_0^t \|V_n\|_{\tau,r,\gamma,r}^p ds\right).
    \end{split}
\end{equation*}
From the assumptions on $\sigma$ in \eqref{noise-viscous}, we know that
\begin{equation*}
\begin{split}
    &\|\sigma(V_n)\|^2_{L_2(\Uc, \mathcal D_0)} + \|A^r e^{\tau A_h}  e^{\gamma A_z} \sigma(V_n)\|^2_{L_2(\Uc, \mathcal D_0)} 
    \\
    \leq &C (1+ \|V_n\|_{\tau,r,\gamma,r}^2 + \|A_h^{\frac12} A^r e^{\tau A_h}  e^{\gamma A_z} V_n\|^2 ) + \delta^2 \|\partial_z V_n\|_{\tau,r,\gamma,r}^2.
\end{split}
\end{equation*}
Therefore, by the Cauchy-Schwarz inequality, 
\begin{equation*}
\begin{split}
    &|A_3|+| A_4| + |B_3|+|B_4| 
    \\
    \leq &C_p  (1+ \|V_n\|_{\tau,r,\gamma,r}^2 + \|A_h^{\frac12} A^r e^{\tau A_h}  e^{\gamma A_z} V_n\|^2  + \delta^2 \|\partial_z V_n\|_{\tau,r,\gamma,r}^2 )\|A^r e^{\tau A_h}  e^{\gamma A_z}V_n\|^{p-2}
    \\
    \leq &C_p(1+ \|V_n\|_{\tau,r,\gamma,r}^p) + C_p \|A^r e^{\tau A_h}  e^{\gamma A_z}V_n\|^{p-2} \left(\|A_h^{\frac12} A^r e^{\tau A_h}  e^{\gamma A_z} V_n\|^2 + \delta^2 \|\partial_zV_n\|_{\tau,r,\gamma,r}^2\right) .
\end{split}
\end{equation*}
Finally, thanks to the Burkholder-Davis-Gundy inequality, from the Cauchy-Schwartz inequality, Young's inequality, and the property of in $\sigma$ \eqref{noise-viscous}, we deduce that
\begin{equation*}
\begin{split}
    &\mathbb E \sup\limits_{s\in[0,t]} \Big|\int_0^s A_5 + B_5 dW \Big| 
    \\
    \leq &C_p \mathbb E\left(\int_0^t \|A^r e^{\tau A_h}  e^{\gamma A_z}V_n\|^{2(p-1)}  \left(1+ \|V_n\|_{\tau,r,\gamma,r}^2 + \|A_h^{\frac12} A^r e^{\tau A_h}  e^{\gamma A_z} V_n\|^2  + \delta^2 \|\partial_z V_n\|_{\tau,r,\gamma,r}^2 \right) \right)^{\frac12}
    \\
    \leq &\frac14 \left(\mathbb E\sup\limits_{s\in[0,t]} \|V_n\|^p + \mathbb E \sup\limits_{s\in[0,t]}\|A^r e^{\tau A_h}  e^{\gamma A_z}  V_n\|^p  \right)
    \\
    &+ C_p\mathbb E \int_0^t \left(1+ \|V_n\|_{\tau,r,\gamma,r}^p + \|A^r e^{\tau A_h}  e^{\gamma A_z}V_n\|^{p-2} \left(\|A_h^{\frac12} A^r e^{\tau A_h}  e^{\gamma A_z} V_n\|^2 +
    \delta^2  \|\partial_zV_n\|_{\tau,r,\gamma,r}^2\right) \right) ds .
\end{split}
\end{equation*}
Notice that since $|k_3|^a\leq |k_3|^b$ for $0\leq a \leq b$ and $k_3\in \Z$, one has
\begin{equation*}
\begin{split}
    \|A^r e^{\tau A_h}  e^{\gamma A_z}V_n\|^{p-2} \|\partial_z V_n\|_{\tau,r,\gamma,r}^2
    &\leq 2 \|A_z A^r  e^{\tau A_h}  e^{\gamma A_z} V_n\|^2 \|A^r e^{\tau A_h}  e^{\gamma A_z} V_n\|^{p-2}.
\end{split}
\end{equation*}
Combining all estimates above produces
\begin{equation}\label{est:1-viscous}
    \begin{split}
        & \mathbb E \sup\limits_{s\in[0,t]} \|V_n\|^p + \mathbb E \sup\limits_{s\in[0,t]} \|A^r e^{\tau A_h}  e^{\gamma A_z} V_n\|^p + \mathbb E \int_0^t \frac p2\nu_z \|A^r e^{\tau A_h}  e^{\gamma A_z}V_n\|^{p-2} \|\partial_z V_n\|_{\tau,r,\gamma,r}^2 ds
        \\
        &+ \mathbb E \int_0^t \|A^r e^{\tau A_h}  e^{\gamma A_z}V_n\|^{p-2} \|A_h^{\frac12} A^r e^{\tau A_h}  e^{\gamma A_z} V_n\|^2 ds
        \\
        \leq & \mathbb E \sup\limits_{s\in[0,t]} \|V_n\|^p + \mathbb E \sup\limits_{s\in[0,t]} \|A^r e^{\tau A_h}  e^{\gamma A_z} V_n\|^p + \mathbb E \int_0^t p\nu_z \|A^r e^{\tau A_h}  e^{\gamma A_z}V_n\|^{p-2}  \| A_z A^{r}  e^{\tau A_h}  e^{\gamma A_z} V_n\|^2 ds
        \\
        &+ \mathbb E \int_0^t \|A^r e^{\tau A_h}  e^{\gamma A_z}V_n\|^{p-2} \|A_h^{\frac12} A^r e^{\tau A_h}  e^{\gamma A_z} V_n\|^2 ds
        \\
        \leq & p \mathbb E \int_0^t \dot{\gamma} \|A^r e^{\tau A_h}  e^{\gamma A_z}V_n\|^{p-2} \| A_z^{\frac12} A^{r}  e^{\tau A_h}  e^{\gamma A_z} V_n\|^2 ds 
        \\
        &+ p \big(\dot{\tau} + C_{p,r,\nu_z} (\rho^2+1)\big) \mathbb E \int_0^t \|A^r e^{\tau A_h}  e^{\gamma A_z}V_n\|^{p-2} \|A_h^{\frac12} A^r e^{\tau A_h}  e^{\gamma A_z} V\|^2 ds 
        \\
        & + \mathbb E \int_0^t  \left(\frac p8\nu_z + C_p \delta^2 \right)\|A^r e^{\tau A_h}  e^{\gamma A_z}V_n\|^{p-2} \|\partial_z V_n\|_{\tau,r,\gamma,r}^2 ds
        \\
        & + \frac12 \left(\mathbb E\sup\limits_{s\in[0,t]} \|V_n\|^p + \mathbb E \sup\limits_{s\in[0,t]}\|A^r e^{\tau A_h}  e^{\gamma A_z}  V_n\|^p  \right)
        + C_p \mathbb E \int_0^t \left(1+ \|V_n\|_{\tau,r,\gamma,r}^p\right) ds.
    \end{split}
\end{equation}
Under the assumption that
\begin{equation}\label{assumption:delta}
    \delta^2 \leq \frac{ p\nu_z}{8 C_p},
\end{equation} 
by taking 
\begin{equation}\label{tau-gamma}
    \tau(t) = - C_{p,r,\nu_z} (\rho^2+1) t, \text{ and } \gamma(t)=\frac18\nu_z t,
\end{equation}
 one can rewrite \eqref{est:1-viscous} as
\begin{equation}\label{est:2-viscous}
    \begin{split}
        & \frac12 \left(\mathbb E\sup\limits_{s\in[0,t]} \|V_n\|^p + \mathbb E \sup\limits_{s\in[0,t]}\|A^r e^{\tau A_h}  e^{\gamma A_z}  V_n\|^p  \right) + \mathbb E \int_0^t \frac p8 \nu_z\|A^r e^{\tau A_h}  e^{\gamma A_z}V_n\|^{p-2} \|\partial_z V_n\|_{\tau,r,\gamma,r}^2 ds
        \\
        &+ \mathbb E \int_0^t \|A^r e^{\tau A_h}  e^{\gamma A_z}V_n\|^{p-2} \|A_h^{\frac12} A^r e^{\tau A_h}  e^{\gamma A_z} V_n\|^2 ds
        \\
        \leq  &C_p \left(1+\mathbb E \int_0^t \left(1+ \|V_n\|_{\tau,r,\gamma,r}^p\right) ds \right).
    \end{split}
\end{equation}
Now define the time $T$ by
\begin{equation}\label{T-viscous}
    T = \frac{\tau_0}{2C_{p,r,\nu_z}(\rho^2 + 1)},
\end{equation}
then one observes that $\tau(T) = \frac{\tau_0}2$ and $\tau(t)>0$ for $t\in[0,T]$. 
Based on \eqref{est:2-viscous}, we have for $t\in[0,T]$,
\begin{equation}\label{est:3-viscous}
     \begin{split}
        & \mathbb E\sup\limits_{s\in[0,t]} \|V_n\|_{\tau,r,\gamma,r}^p  + \mathbb E \int_0^t \nu_z \|A^r e^{\tau A_h}  e^{\gamma A_z}V_n\|^{p-2} \|\partial_z V_n\|_{\tau,r,\gamma,r}^2 ds
        \\
        &+ \mathbb E \int_0^t \|A^r e^{\tau A_h}  e^{\gamma A_z}V_n\|^{p-2} \|A_h^{\frac12} A^r e^{\tau A_h}  e^{\gamma A_z} V_n\|^2 ds
        \\
        \leq & C_p \left(1+\mathbb E \int_0^t \left(1+ \|V_n\|_{\tau,r,\gamma,r}^p\right) ds \right).
    \end{split}
\end{equation}
Thanks to the Gronwall inequality \cite[Lemma 5.3]{glatt2009strong}, for $p\geq 2$ and $t=T$, one has
\begin{equation}\label{est:viscous}
\begin{split}
     \mathbb E \Big[\sup\limits_{t\in[0,T]} &\|V_n\|_{\tau,r,\gamma,r}^p +  \int_0^T \nu_z \|A^r e^{\tau A_h}  e^{\gamma A_z}V_n\|^{p-2} \|\partial_z V_n\|_{\tau,r,\gamma,r}^2 ds
    \\
    &+ \int_0^T \|A^r e^{\tau A_h}  e^{\gamma A_z}V_n\|^{p-2} \|A_h^{\frac12} A^r e^{\tau A_h}  e^{\gamma A_z} V_n\|^2 ds \Big]  C_p\leq (1+ \mathbb E \|V_0\|_{\tau_0,r,0,r}^p) e^{C_p T}.
\end{split}
\end{equation}

Regarding part (ii), we apply the fractional Burkholder-Davis-Gundy inequality \eqref{eq:bdg.frac}, \eqref{noise-viscous}, and the estimate \eqref{est:viscous} to obtain 
\begin{equation*}
    \begin{split}
        &\Eb \left| \int_0^\cdot e^{\tau A_h }e^{\gamma A_z}\sigma_n\left(V_n\right) \, dW \right|^p_{W^{\alpha, p}\left(0, T; \mathcal D_{0,r-1}\right)} 
        \\
        \leq &C_p \Eb \int_0^T \| e^{\tau A_h }e^{\gamma A_z}\sigma_n\left(V_n\right) \|^p_{L_2\left(\Uc, \mathcal D_{0,r-1}\right)} \, ds 
        \leq  C_{p,\delta} \Eb \left[  \int_0^T 1 + \|V_n\|_{\tau,r,\gamma,r}^p  \, ds \right] < \infty.
    \end{split}
\end{equation*}

For part (iii), from \eqref{PE-viscous-galerkin-1} (or \eqref{PE-viscous-galerkin-U-1}), for $t\in[0,T]$, one has
\begin{equation*}
    \begin{split}
        &e^{\tau A_h }e^{\gamma A_z}V_n(t) - \int_0^t e^{\tau A_h }e^{\gamma A_z}\sigma_n\left(V_n\right) \, dW 
        \\
        =& e^{\tau_0 A_h } V_n(0) - \int_0^t \Big[\theta_\rho(\|V_n\|_{\tau,r,\gamma,r} ) e^{\tau A_h }e^{\gamma A_z} Q_n(V_n, V_n) + e^{\tau A_h }e^{\gamma A_z} F_n\left(V_n\right) - \nu_z e^{\tau A_h }e^{\gamma A_z} \partial_{zz} V_n
        \\
        &\qquad \qquad \qquad \qquad - \dot{\tau}A_h e^{\tau A_h} e^{\gamma A_z} V_n -  \dot{\gamma}A_z e^{\tau A_h} e^{\gamma A_z} V_n \Big]\, ds.
    \end{split}
\end{equation*}
Then, with Lemma~\ref{lemma-banach-algebra-viscous} and $r>\frac52$, the Minkowski inequality, and Young's inequality, and the properties of $\sigma$, $F$ and the cut-off function $\theta_\rho$, we have
\begin{equation}\label{estimate:w12-viscous}
\begin{split}
    & \Eb \left\| e^{\tau A_h }e^{\gamma A_z}V_n - \int_0^\cdot e^{\tau A_h }e^{\gamma A_z} \sigma_n\left(V_n\right) \, dW \right\|^2_{W^{1, 2}\left(0, T; \mathcal D_{0,r-1}\right)} 
    \\
    \leq &C_{T,f_0,p,\rho,r,\nu_z} \Eb \left[ 1+ \| V_0\|_{\tau_0,r,0,r}^2 + \sup\limits_{t\in[0,T]}\|V_n\|_{\tau,r,\gamma,r}^4+ \int_0^T   \|f\|_{\tau_0,r,\gamma(T),r}^2 + \|\partial_z V_n\|_{\tau,r,\gamma,r}^2 \, ds \right]
    < \infty,
\end{split}
\end{equation}
where the pressure gradient disappears since $\nabla P$ is orthogonal to the space $\mathcal D_{0,r-1}$. Again, we require a higher moment $p \geq 4$ here due to applying item (i) to $\|V_n\|_{\tau,r,\gamma,r}^4$ in \eqref{estimate:w12-viscous}.
\end{proof}

\begin{corollary}\label{cor:u-viscous}
With the same assumptions as in Lemma~\ref{lemma:estimate-inviscid}. Let $\tau(t)$ and $\gamma(t)$ be defined in \eqref{tau-gamma}, and let $T$ be defined in \eqref{T-viscous}.
Suppose that $U_n = e^{\tau A_h }e^{\gamma A_z}V_n$ be the solutions of system \eqref{PE-viscous-system-galerkin-U}, and we denote by $U_0 = e^{\tau_0 A_h } V_0$. 
Then
	\begin{multline*}
\text{(i)} \quad     \mathbb E \Big[\sup\limits_{t\in[0,T]} \|U_n\|_{0,r}^p +  \int_0^T \nu_z \|A^r U_n\|^{p-2} \|\partial_z U_n\|_{0,r}^2 ds
    \\
    + \int_0^T \|A^r U_n\|^{p-2} \|A_h^{\frac12} A^r U_n\|^2 ds \Big]  \leq C_p (1+ \mathbb E \|U_0\|_{0,r}^p) e^{C_p T}.
\end{multline*}

		(ii) For  $\alpha\in [0,1/2)$, $\int_0^{\cdot} e^{\tau A_h }e^{\gamma A_z}\sigma_n(e^{-\tau A_h }e^{-\gamma A_z}U_n) dW$ is bounded in
	$
		L^p\left( \Omega; W^{\alpha, p}(0, T; \mathcal D_{0,r-1}) \right);	
	$

	(iii) If $p\geq 4$, $U_n - \int_0^\cdot e^{\tau A_h }e^{\gamma A_z}\sigma_n(e^{-\tau A_h }e^{-\gamma A_z}U_n) dW$ is bounded in
	$
		L^2\left( \Omega; W^{1, 2}(0, T; \mathcal D_{0,r-1}))\right) .
	$
\end{corollary}

\subsection{Local existence of martingale solution and pathwise uniqueness}\label{section:existence-uniqueness-viscous}

This section follows similar to Sections~\ref{section:existence-inviscid} and \ref{section:uniqueness-inviscid}, once energy estimates are obtained above. Thus we give detail only when necessary.

Recall that when $\tau = \gamma =0$ and $r=s$, $\mathcal D_{0,r,0,s} = \mathcal D_{0,r}$. Given an initial distribution $\mu_0$ satisfying \eqref{eq:mu.zero-viscous} for some $p\geq 4$ (which is implied by \eqref{condition:mu-zero-viscous}), for some stochastic basis $\Sc = \left(\Omega, \Fc, \Fb, \Pb\right)$, let $U_0= e^{\tau_0 A_h} V_0$ be an $\Fc_0$-measurable $\mathcal D_{0,r}$-valued random variable with law $\mu_0$. With the same settings as in Section~\ref{section:existence-inviscid}, we have the following two propositions similar to Proposition~\ref{prop:approximating.sequence} and Proposition~\ref{prop:global.martingale.existence}

\begin{proposition}
\label{prop:approximating.sequence-viscous}
Let $\mu_0$ be a probability measure on $\mathcal D_{0,r}$ satisfying 
\begin{equation*}
    \int_{\mathcal D_{0,r}} \|U\|_{0,r}^p d\mu_0(U)<\infty 
\end{equation*}
with $p\geq4$ and let $(\mu^n)_{n \geq 1}$ be the measures defined in \eqref{eq:mu.measure}. Then there exists a probability space $(\tilde{\Omega}, \tilde{\Fc}, \tilde{\Pb})$, a subsequence $n_k \to \infty$ as $k \to \infty$ and a sequence of $\Xc$-valued random variables $(\tilde{U}_{n_k}, \tilde{W}_{n_k})$ such that
\begin{enumerate}
	\item $(\tilde{U}_{n_k}, \tilde{W}_{n_k})$ converges in $\Xc$ to $(\tilde{U}, \tilde{W}) \in \Xc$ almost surely,
	\item $\tilde{W}_{n_k}$ is a cylindrical Wiener process with reproducing kernel Hilbert space $\Uc$ adapted to the filtration $\left( \Fc_t^{n_k} \right)_{t \geq 0}$, where $\left( \Fc_t^{n_k} \right)_{t \geq 0}$ is the completion of $\sigma(\tilde{W}_{n_k}, \tilde{U}_{n_k}; s \leq t)$, 
	\item for $t\in[0,T]$, each pair $(\tilde{U}_{n_k}, \tilde{W}_{n_k})$ satisfies the equation
	\begin{equation}
	\label{eq:appr.after.skorohod-viscous}
	\begin{split}
	    &d \tilde U_{n_k} + \big[\theta_\rho(\|\tilde U_{n_k}\|_{0,r} ) e^{\tau A_h} e^{\gamma A_z} Q_{n_k}(e^{-\tau A_h} e^{-\gamma A_z}\tilde U_{n_k}, e^{-\tau A_h} e^{-\gamma A_z}\tilde U_{n_k}) 
	    \\
	    &\qquad+ e^{\tau A_h} e^{\gamma A_z}F_{n_k}(e^{-\tau A_h} e^{-\gamma A_z}\tilde U_{n_k}) 
	    -\dot{\tau}A_h \tilde U_{n_k}  - \dot\gamma A_z \tilde U_{n_k} - \nu_z \partial_{zz} \tilde U_{n_k}\big]dt  
	    \\
	    = &e^{\tau A_h} e^{\gamma A_z}\sigma_{n_k} (e^{-\tau A_h} e^{-\gamma A_z}\tilde U_{n_k}) d\tilde W_{n_k}, 
	\end{split}
	\end{equation}
	where $\tau$ and $\gamma$ are defined in \eqref{tau-gamma} and $T$ is defined in \eqref{T-viscous}.
\end{enumerate}
\end{proposition}
The proof of Proposition~\ref{prop:approximating.sequence-viscous} is almost identical to that of Proposition~\ref{prop:approximating.sequence}, so we omit the details.

\begin{proposition}
\label{prop:global.martingale.existence-viscous}
For $\tau(t), \gamma(t)$ defined in \eqref{tau-gamma} and $T$ defined in \eqref{T-viscous}, let $(\tilde{U}_{n_k}, \tilde{W}_{n_k})$ be a sequence of $\Xc$-valued random variables on a probability space $(\tilde{\Omega}, \tilde{\Fc}, \tilde{\Pb})$ such that
\begin{enumerate}
	\item $(\tilde{U}_{n_k}, \tilde{W}_{n_k}) \to (\tilde{U}, \tilde{W})$ in the topology of $\Xc$, $\tilde{\Pb}$-almost surely, that is,
	\begin{equation*}
		\tilde{U}_{n_k} \to \tilde{U} \ \text{in} \ L^2\left(0, T; \mathcal D_{0,r}\right) \cap C\left(\left[0, T \right], \mathcal D_{0,r-\frac32}\right), \ \tilde{W}_{n_k} \to \tilde{W} \ \text{in} \ C\left(\left[0, T\right]; \Uc_0 \right), 
	\end{equation*}
	\item $\tilde{W}_{n_k}$ is a cylindrical Wiener process with reproducing kernel Hilbert space $\Uc$ adapted to the filtration $\left( \Fc_t^{n_k} \right)_{t \geq 0}$ that contains $\sigma(\tilde{W}_{n_k}, \tilde{U}_{n_k}; s \leq t)$,
	\item each pair $(\tilde{U}_{n_k}, \tilde{W}_{n_k})$ satisfies \eqref{eq:appr.after.skorohod-viscous}.
\end{enumerate}
Let 
\begin{equation*}
    \tilde V_{n_k} = e^{-\tau A_h} e^{-\gamma A_z} \tilde U_{n_k}, \ \ \tilde V = e^{-\tau A_h} e^{-\gamma A_z} \tilde U,
\end{equation*}
and let $\tilde{\Fc}_t$ be the completion of $\sigma(\tilde{W}(s), \tilde{V}(s), 0 \leq s \leq t)$ and $\tilde{\Sc} = (\tilde{\Omega}, \tilde{\Fc}, ( \tilde{\Fc}_t )_{t \geq 0}, \tilde{\Pb})$. Then $(\tilde{\Sc}, \tilde{W}, \tilde{V},T)$ is a local martingale solution to the modified system \eqref{PE-viscous-system-modified} on the time interval $[0,T]$. Moreover, the solution $\tilde{V}$ satisfies
\begin{equation}
	\label{eq:martingale.solution.approx.regularity-viscous}
	\begin{split}
	    &\tilde{V} \in L^2\left( \tilde \Omega; C\left( [0, T]; \mathcal D_{\tau(t),r,\gamma(t),r} \right) \right), 
	    \\
	    &\left(\|A_h^{\frac12} A^r e^{\tau(t)A_h } e^{\gamma(t) A_z} \tilde{V} \|^2 + \|A_z A^r e^{\tau(t)A_h } e^{\gamma(t) A_z} \tilde{V} \|^2 \right)   \|A^{r} e^{\tau(t)A_h } e^{\gamma(t) A_z} \tilde{V}\|^{p-2}  \in L^1\left(\tilde \Omega; L^1\left( 0, T \right) \right).
	\end{split}
\end{equation}
\end{proposition}

\begin{proof}[Skecth of proof]
The proof of $(\tilde{\Sc}, \tilde{W}, \tilde{V},T)$ being a local martingale solution to the modified system \eqref{PE-viscous-system-modified} is very similar to the proof in Proposition~\ref{prop:global.martingale.existence}. Heuristically, the main differences between the inviscid case and the vertically viscous case are on the equipped norm and the additional vertical viscosity term. The difference on the norm is that in the viscous case the horizontal and vertical radii of analyticity are different. However, from Lemma~\ref{lemma:estimate-viscous} we see this will not give any trouble. The convergence of the vertical viscosity term is straightforward thanks to the regularity of the solution. 

We will only provide a detailed proof of \eqref{eq:martingale.solution.approx.regularity-viscous}. By the growth condition of $\sigma$ assumed in \eqref{noise-viscous} and the regularity of $\tilde V$ from Lemma~\ref{lemma:estimate-viscous}, we have
\[
	\sigma(\tilde V) \in L^2\left(\tilde{\Omega}; L^2\left(0,T; L_2\left(\Uc, \mathcal D_{\tau(t),r,\gamma(t),r}\right)\right)\right).
\]
Therefore, the solution to
\[
	dZ - \nu_z \partial_{zz} Z  = \sigma(\tilde V) \, d\tilde{W}, \qquad Z(0) = \tilde V_0,
\]
satisfies
\begin{equation}
	\label{eq:z.regularity-viscous}
	Z \in L^2\left(\tilde{\Omega}; C\left([0,T]; \mathcal D_{\tau(t),r,\gamma(t),r}\right)\right), \qquad A^r e^{\tau A_h}e^{\gamma A_z} Z \in L^2\left(\tilde{\Omega}; L^2\left(0,T; \mathcal D_{0,\frac12,0,1}\right)\right).
\end{equation}
Similar to \eqref{bound:v-and-vtilde}, we have the following regularity for $\tilde{V}$:
\begin{equation}\label{bound:v-and-vtilde-viscous}
\begin{split}
    &\tilde{V} \in L^2 \left( \tilde \Omega; L^\infty\big( 0, T; \mathcal D_{\tau(t),r,\gamma(t),r}   \big )  \right), \qquad A^r e^{\tau A_h}e^{\gamma A_z} \tilde{V} \in L^2\left(\tilde{\Omega}; L^2\left(0,T; \mathcal D_{0,\frac12,0,1}\right)\right).
\end{split}
\end{equation}

Defining $\bar{V} = \tilde V - Z$, by \eqref{PE-viscous-modified-1} we have $\Pb$-almost surely
\begin{equation*}
	\tfrac{d}{dt} \bar{V}  + \theta(\| \tilde V  \|_{\tau,r,\gamma,r}) Q(\tilde V, \tilde V) + F(\tilde V) - \nu_z \partial_{zz} \bar{V}= 0, \qquad \bar{V}(0) = 0.
\end{equation*}
Notice that
\begin{equation*}
    \frac{d}{dt} A^{r}e^{\tau(t)A_h} e^{\gamma(t) A_z}\bar{V} = A^{r}e^{\tau(t)A_h} e^{\gamma(t) A_z}  \frac{d}{dt} \bar{V} + \dot{\tau}A_h A^r e^{\tau(t)A_h} e^{\gamma(t) A_z} \bar{V} + \dot{\gamma}A_z A^r e^{\tau(t)A_h} e^{\gamma(t) A_z} \bar{V},
\end{equation*}
thanks to \eqref{eq:z.regularity-viscous}, \eqref{bound:v-and-vtilde-viscous}, and Lemma~\ref{lemma-banach-algebra-viscous}, for arbitrary $\phi\in \mathcal D_{0,\frac12,0,1}$, one has 
\[
 \Big\langle \frac{d}{dt} A^{r}e^{\tau(t)A_h} e^{\gamma(t) A_z}\bar{V}, \phi \Big\rangle \in L^2\left(\tilde{\Omega}; L^2\left(0,T\right)\right).
\]
Therefore,
\[
\frac{d}{dt} A^{r}e^{\tau(t)A_h} e^{\gamma(t) A_z}\bar{V} \in L^2\left(\tilde{\Omega}; L^2\left(0,T; \mathcal D_{0,\frac12,0,1}'\right)\right),
\]
where $\mathcal D_{0,\frac12,0,1}'$ is the dual space of $\mathcal D_{0,\frac12,0,1}$. Since $\mathcal D_{0,\frac12,0,1} \subset \mathcal D_0 \equiv \mathcal D_0' \subset \mathcal D_{0,\frac12,0,1}'$, by the Lions-Magenes Lemma, see, {\it e.g.}, \cite[Lemma 1.2, Chapter 3]{temam2001navier}, we infer that $A^{r}e^{\tau(t)A_h} e^{\gamma(t) A_z} \bar{V} \in L^2\left( \tilde \Omega; C\left(\left[0,T \right]; \mathcal D_0\right)\right)$. Similarly, one can show that $\bar{V} \in L^2\left( \tilde \Omega; C\left(\left[0,T \right]; \mathcal D_0\right)\right)$, and thus $\bar{V} \in L^2\left( \tilde \Omega; C\left(\left[0,T \right]; \mathcal D_{\tau(t),r,\gamma(t),r}\right)\right)$. This together with \eqref{eq:z.regularity-viscous} imply that $\tilde{V} \in L^2\left( \tilde \Omega; C\left(\left[0,T \right]; \mathcal D_{\tau(t),r,\gamma(t),r}\right)\right)$. The second part of \eqref{eq:martingale.solution.approx.regularity-viscous} follows directly from \eqref{bound:v-and-vtilde-viscous}.

\end{proof}

\begin{corollary}
\label{cor:loc.mart.sol-viscous}
Suppose that $\mu_0$ satisfy \eqref{condition:mu-zero-viscous} with constant $M>0$. Let $\rho \geq M$, and let $(\tilde{\Sc}, \tilde{W}, \tilde{V},T)$ be the local martingale solution to system \eqref{PE-viscous-system-modified} given in Proposition~\ref{prop:global.martingale.existence-viscous}. Let
\begin{equation}\label{stopingtime:eta-viscous}
	\eta = \inf \left\lbrace t \geq 0 \mid \| \tilde{V}\|_{\tau,r,\gamma,r} \geq \frac\rho2 \right\rbrace,
\end{equation}
Then $(\tilde{\Sc}, \tilde{W}, \tilde{V}, \eta\wedge T)$ is a local martingale solution to the problem \eqref{PE-viscous-system-abstract}.
Moreover, 
\begin{equation*}
	\begin{split}
	    &\tilde{V}\left( \cdot \wedge \eta \right) \in L^2 \left( \Omega; C\left( [0, T], \mathcal D_{\tau(t),r,\gamma(t),r} \right) \right),
	    \\
	    &\mathds{1}_{[0, \eta]} \left(\|A_h^{\frac12} A^r e^{\tau(t)A_h } e^{\gamma(t) A_z} \tilde{V} \|^2 + \|A_z A^r e^{\tau(t)A_h } e^{\gamma(t) A_z} \tilde{V} \|^2 \right)   
	    \\
	    &\qquad \qquad \qquad\times \|A^{r} e^{\tau(t)A_h } e^{\gamma(t) A_z} \tilde{V}\|^{p-2}\in L^1\left(\tilde \Omega; L^1\left( 0, T \right) \right).
	\end{split}
\end{equation*}
\end{corollary}

Finally, we state the pathwise uniqueness for the original system \eqref{PE-viscous-system-abstract}. 

\begin{proposition}
\label{prop:pathwise.uniqueness-viscous}
Let $\tau(t), \gamma(t)$ and $T$ defined as in \eqref{tau-gamma} and \eqref{T-viscous}, respectively. Suppose that $\sigma$ and $f$ satisfy \eqref{noise-viscous} and \eqref{force-viscous}, respectively, with $\delta$ satisfying a similar smallness condition as \eqref{assumption:delta}. Let $\Sc = \left(\Omega, \Fc, \Fct, \Pb \right)$ and $W$ be fixed. Suppose that there exist two local martingale solutions  $\left(\Sc, W, V^1,T\right)$ and $\left(\Sc, W, V^2,T\right)$ to the modified system \eqref{PE-viscous-system-modified}. Correspondingly,   $\left(\Sc, W, V^1, \eta_1\wedge T\right)$ and $\left(\Sc, W, V^2, \eta_2\wedge T\right)$ are the two local martingale solutions to the original system \eqref{PE-viscous-system-abstract}. Denote $\Omega_0 = \left\lbrace V^1(0) = V^2(0) \right\rbrace \subseteq \Omega$, and  $\eta = \eta_1 \wedge \eta_2$. Then
\begin{equation*}
	\Pb \left( \left\lbrace \mathds{1}_{\Omega_0}\left(V^1(t\wedge \eta) - V^2(t \wedge \eta)\right) = 0 \ \text{for all} \ t \in[0,T] \right\rbrace \right) = 1.
\end{equation*}
\end{proposition}
The proof is very similar to that of Proposition~\ref{prop:pathwise.uniqueness}, and thus we omit it.

\section{Conclusive remarks}\label{sec:conclusion}
We prove the local existence of martingale solutions and pathwise uniqueness for the 3D stochastic inviscid primitive equations (PEs, or hydrostatic Euler equations), where we consider a larger class of noises than multiplicative noises, and work in the analytic function space due to the ill-posedness in Sobolev spaces of PEs without horizontal viscosity. By adding vertical viscosity, we can relax the restriction on initial conditions to be only analytic in the horizontal variables with Sobolev regularity in the vertical variable, and allow the transport noise in the vertical direction. Moreover, the solution becomes analytic in $z$ instantaneously as $t>0$ and the analytic radius in $z$ increases as long as the solutions exist. The global in time existence of solutions is not achieved, as the parallel results in the deterministic case are still open in the vertically viscous case, and not true in the inviscid case. The existence of pathwise solutions still remains open, since the nonlinear estimates in the analytic function space are not as good as the horizontal viscosity case in Sobolev spaces. We shall leave it as future work.

\noindent
\section*{Acknowledgments}
R.H.~was partially supported by the NSF grant DMS-1953035, the Faculty Career Development Award, the Research Assistance Program Award, and the
Early Career Faculty Acceleration funding at University of California, Santa Barbara. 

\appendix
\section{Estimates of nonlinear terms}
In this appendix, we provide estimates of nonlinear terms in the analytic function space which has been used throughout the paper. Recall that $Q(f,g) = f\cdot \nabla g -\int_0^z (\nabla \cdot  f) (\boldsymbol{x}', \tilde z)  d\tilde z \, \partial_z g$, and for $r,\tau,s,\gamma \geq 0$,
    \begin{equation*}
    \begin{split}
        &A^r f := \sum\limits_{\bk\in  2\pi\mathbb{Z}^3} |\bk|^r   \hat{f}_{\bk} e^{ i \bk\cdot\bx}, \qquad A_h^r f := \sum\limits_{\bk\in  2\pi\mathbb{Z}^3} |\bk'|^r   \hat{f}_{\bk} e^{ i \bk\cdot\bx}, \qquad A_z^s f := \sum\limits_{\bk\in  2\pi\mathbb{Z}^3} |k_3|^s  \hat{f}_{\bk} e^{ i \bk\cdot\bx},
        \\
     &e^{\tau A} f := \sum\limits_{\bk\in  2\pi\mathbb{Z}^3}  e^{\tau(t)|\bk|}  \hat{f}_{\bk} e^{ i \bk\cdot\bx}, \qquad e^{\tau A_h} e^{\gamma A_z} f := \sum\limits_{\bk\in  2\pi\mathbb{Z}^3}  e^{\tau(t)|\bk'|} e^{\gamma(t)|k_3|}  \hat{f}_{\bk} e^{ i \bk\cdot\bx},
    \end{split}
\end{equation*}
where $\bk = (\bk', k_3)$.

The first lemma is used in the inviscid case.
\begin{lemma}[{\cite[Lemmas A.1 and A.3]{ghoul2022effect}}]\label{lemma-inviscid}
For $f, g, h\in \mathcal{D}_{\tau,r+\frac12}$, where $r>2$ and $\tau\geq 0$, one has
\begin{equation*}\label{lemma-inviscid-inequality}
    \begin{split}
        \Big|\Big\langle A^r e^{\tau A} Q(f,g), A^r e^{\tau A} h  \Big\rangle\Big| \leq  C_r\Big( &\|f\|_{\tau,r} \|A^{r+\frac{1}{2}} e^{\tau A} g\| \|A^{r+\frac{1}{2}} e^{\tau A} h\|  + \|g\|_{\tau,r} \|A^{r+\frac{1}{2}} e^{\tau A} f\| \|A^{r+\frac{1}{2}} e^{\tau A} h\| 
        \\
       &+ \| h\|_{\tau,r} \|A^{r+\frac{1}{2}} e^{\tau A} f\| \|A^{r+\frac{1}{2}} e^{\tau A} g\|\Big).
    \end{split}
\end{equation*}
\end{lemma}

The next lemma is used in the vertically viscous case.
\begin{lemma}\label{lemma-viscous-type1}
For $f, g, h\in \mathcal{D}_{\tau,r+\frac12,\gamma,r} \cap \mathcal{D}_{\tau,r,\gamma,r+1}$, where $r>2$, $\tau,\gamma\geq 0$, and $\int_0^1 \nabla \cdot f(\boldsymbol{x}', z)dz = 0$, one has
\begin{equation*}\label{lemma-viscous-type1-inequality}
    \begin{split}
       &\Big|\Big\langle A^r e^{\tau A_h} e^{\gamma A_z} Q(f,g), A^r e^{\tau A_h} e^{\gamma A_z} h  \Big\rangle\Big|
       \\
       \leq &C_r\Big[\|f\|_{\tau,r,\gamma,r} \|A_h^{\frac12} A^r e^{\tau A_h} e^{\gamma A_z} g\| \|A_h^{\frac12} A^r e^{\tau A_h} e^{\gamma A_z} h\|
     + \|g\|_{\tau,r,\gamma,r} \|A_h^{\frac12} A^r e^{\tau A_h} e^{\gamma A_z} f\|  \|A_h^{\frac12} A^r e^{\tau A_h} e^{\gamma A_z} h\|
     \\
     & \quad+ \|h\|_{\tau,r,\gamma,r} \|A_h^{\frac12} A^r e^{\tau A_h} e^{\gamma A_z} f\|  \|A_h^{\frac12} A^r e^{\tau A_h} e^{\gamma A_z} g\|\Big]
     \\
     &+ C_r \| A^{r} e^{\tau A_h} e^{\gamma A_z} f\| \left(\|A_h^{\frac{1}{2}} A^{r} e^{\tau A_h} e^{\gamma A_z} g\| \|A_z A^{r} e^{\tau A_h} e^{\gamma A_z} h\|+\|A_z A^{r} e^{\tau A_h} e^{\gamma A_z} g\| \|A_h^{\frac{1}{2}} A^{r} e^{\tau A_h} e^{\gamma A_z} h\| \right).
    \end{split}
\end{equation*}
\end{lemma}
\begin{proof}
The proof is similar to \cite[Lemma A.1 and Lemma A.3]{ghoul2022effect} with some modifications. For completeness, we present it below but provide only necessary details. First, since $Q(f,g) = f\cdot \nabla g -\int_0^z (\nabla \cdot  f) (\tilde z)  d\tilde z \, \partial_z g$, one has
\begin{equation*}
    \begin{split}
       &\Big|\Big\langle A^r e^{\tau A_h} e^{\gamma A_z} Q(f,g), A^r e^{\tau A_h} e^{\gamma A_z} h  \Big\rangle\Big|
       \\
       \leq & \Big|\Big\langle A^r e^{\tau A_h} e^{\gamma A_z} \left(f\cdot \nabla g\right), A^r e^{\tau A_h} e^{\gamma A_z} h  \Big\rangle\Big| + \Big|\Big\langle A^r e^{\tau A_h} e^{\gamma A_z} \left(\int_0^z (\nabla \cdot  f) (\tilde z)  d\tilde z \, \partial_z g\right), A^r e^{\tau A_h} e^{\gamma A_z} h  \Big\rangle\Big|
       \\
       := &I_1 + I_2.
    \end{split}
\end{equation*}

For $I_1$, follows almost exactly the proof of \cite[Lemma A.1]{ghoul2022effect}, we have
\begin{equation*}
\begin{split}
     I_1 \leq &C_r\Big(\|f\|_{\tau,r,\gamma,r} \|A_h^{\frac12} A^r e^{\tau A_h} e^{\gamma A_z} g\| \|A_h^{\frac12} A^r e^{\tau A_h} e^{\gamma A_z} h\|
     \\
     & \quad+ \|g\|_{\tau,r,\gamma,r} \|A_h^{\frac12} A^r e^{\tau A_h} e^{\gamma A_z} f\|  \|A_h^{\frac12} A^r e^{\tau A_h} e^{\gamma A_z} h\|
     \\
     & \quad+ \|h\|_{\tau,r,\gamma,r} \|A_h^{\frac12} A^r e^{\tau A_h} e^{\gamma A_z} f\|  \|A_h^{\frac12} A^r e^{\tau A_h} e^{\gamma A_z} g\|\Big).
\end{split}
\end{equation*}

For $I_2$, we use Fourier representation of $f, g$ and $H := A^r e^{\tau A_h} e^{\gamma A_z} h$, in which we can write
    \begin{eqnarray*}
    &&\hskip-.8in
     f(\boldsymbol{x}) = \sum\limits_{\boldsymbol{j}\in  2\pi\mathbb{Z}^3} \hat{f}_{\boldsymbol{j}} e^{ i\boldsymbol{j}\cdot \boldsymbol{x}}, \quad
    g(\boldsymbol{x}) = \sum\limits_{2\pi\boldsymbol{k}\in \mathbb{Z}^3} \hat{g}_{\boldsymbol{k}} e^{ i\boldsymbol{k}\cdot \boldsymbol{x}}, \quad
     H(\boldsymbol{x}) = \sum\limits_{2\pi\boldsymbol{l}\in \mathbb{Z}^3} |\boldsymbol{l}|^r e^{\tau |\boldsymbol{l}'|} e^{\gamma |l_3|}\hat{h}_{\boldsymbol{l}} e^{ i\boldsymbol{l}\cdot \boldsymbol{x}}. 
    \end{eqnarray*}
Since $\int_0^1 \nabla \cdot f(\boldsymbol{x}', z)dz = 0$, one has
$
    f(\boldsymbol{x}) = \sum\limits_{\substack{\boldsymbol{j}\in \mathbb{Z}^3 \\ j_3 \neq 0}} \hat{f}_{\boldsymbol{j}} e^{2\pi( i\boldsymbol{j}' \cdot \boldsymbol{x}' +  ij_3 z)} 
$, and
\begin{equation*}
  \int_0^z \nabla \cdot f(\boldsymbol{x}',\tilde z) d\tilde z = \sum\limits_{\substack{\boldsymbol{j}\in \mathbb{Z}^3 \\ j_3 \neq 0, \boldsymbol{j}' \neq 0}} \frac{1}{j_3}\boldsymbol{j}'\cdot\hat{f}_{\boldsymbol{j}} e^{2\pi( i\boldsymbol{j}' \cdot \boldsymbol{x}' +  ij_3 z)} - \sum\limits_{\substack{\boldsymbol{j}\in \mathbb{Z}^3 \\ j_3 \neq 0, \boldsymbol{j}' \neq 0}} \frac{1}{j_3}\boldsymbol{j}'\cdot\hat{f}_{\boldsymbol{j}} e^{2\pi i\boldsymbol{j}'\cdot \boldsymbol{x}'}.  
\end{equation*}
Therefore,
\begin{equation*}
    \begin{split}
        I_2 = &\Big|\Big\langle A^r e^{\tau A_h} e^{\gamma A_z} \left(\int_0^z (\nabla \cdot  f) (\tilde z)  d\tilde z \, \partial_z g\right), H  \Big\rangle\Big|
        \\
        \leq & \Big|\Big\langle A^r e^{\tau A_h} e^{\gamma A_z} \Big(\sum\limits_{\substack{\boldsymbol{j}\in \mathbb{Z}^3 \\ j_3 \neq 0, \boldsymbol{j}' \neq 0}} \frac{1}{j_3}\boldsymbol{j}'\cdot\hat{f}_{\boldsymbol{j}} e^{2\pi( i\boldsymbol{j}' \cdot \boldsymbol{x}' +  ij_3 z)} \, \partial_z g\Big), H  \Big\rangle\Big| 
        \\
        &+ \Big|\Big\langle A^r e^{\tau A_h} e^{\gamma A_z} \Big(\sum\limits_{\substack{\boldsymbol{j}\in \mathbb{Z}^3 \\ j_3 \neq 0, \boldsymbol{j}' \neq 0}} \frac{1}{j_3}\boldsymbol{j}'\cdot\hat{f}_{\boldsymbol{j}} e^{2\pi i\boldsymbol{j}'\cdot \boldsymbol{x}'}\, \partial_z g\Big), H  \Big\rangle\Big| := I_{21} + I_{22}.
    \end{split}
\end{equation*}
We estimate $I_{22}$ first. For $\boldsymbol{l} = (\boldsymbol{l}', l_3) = (-\boldsymbol{j}'-\boldsymbol{k}', -k_3)$, by using the inequalities 
\begin{equation*}
    |\boldsymbol{j}'|^{\frac{1}{2}} \leq C(|\boldsymbol{k}'|^{\frac{1}{2}} + |\boldsymbol{l}'|^{\frac{1}{2}}),  \;\;  |\boldsymbol{l}|^r \leq C_r (|\boldsymbol{j}|^r+|\boldsymbol{k}|^r), 
\end{equation*}
one has
\begin{eqnarray*}
&&\hskip-.8in 
I_{22} \leq \sum\limits_{\substack{\boldsymbol{j}'+\boldsymbol{k}'+\boldsymbol{l}'=0\\  k_3+l_3 = 0 \\  j_3, k_3, \boldsymbol{j}' \neq 0}} C_r\frac{1}{|j_3|}|\boldsymbol{j}'||k_3|e^{\tau |\boldsymbol{j}'|} e^{\tau |\boldsymbol{k}'|} e^{\gamma |k_3|} |\hat{f}_{\boldsymbol{j}}||\hat{g}_{\boldsymbol{k}}|(|\boldsymbol{j}|^r+|\boldsymbol{k}|^r)|\boldsymbol{l}|^r e^{\tau |\boldsymbol{l}'|} e^{\gamma |l_3|}|\hat{h}_{\boldsymbol{l}}|\nonumber\\
&&\hskip-.68in 
\leq  \sum\limits_{\substack{\boldsymbol{j}'+\boldsymbol{k}'+\boldsymbol{l}'=0\\  k_3+l_3 = 0 \\  j_3, k_3, \boldsymbol{j}' \neq 0}} C_r \frac{|k_3|}{|j_3|}\Big(|\boldsymbol{k}'|^{\frac{1}{2}} |\boldsymbol{j}'|^{\frac{1}{2}}|\boldsymbol{j}|^{r} |\boldsymbol{l}|^{r} + |\boldsymbol{j}'|^{\frac{1}{2}}|\boldsymbol{j}|^{r} |\boldsymbol{l}'|^{\frac{1}{2}} |\boldsymbol{l}|^{r} + |\boldsymbol{j}'|^{\frac{1}{2}} |\boldsymbol{k}'|^{\frac{1}{2}}|\boldsymbol{k}|^{r} |\boldsymbol{l}|^{r}\nonumber\\
&&\hskip-.38in 
 + |\boldsymbol{j}'|^{\frac{1}{2}}|\boldsymbol{k}|^{r} |\boldsymbol{l}'|^{\frac{1}{2}}|\boldsymbol{l}|^{r} \Big) e^{\tau |\boldsymbol{j}'|} e^{\gamma |j_3|}e^{\tau |\boldsymbol{k}'|} e^{\gamma |k_3|} e^{\tau |\boldsymbol{l}'|} e^{\gamma |l_3|} |\hat{f}_{\boldsymbol{j}}| |\hat{g}_{\boldsymbol{k}}| |\hat{h}_{\boldsymbol{l}}| =: B_1 + B_2 + B_3 + B_4.
\end{eqnarray*}
When $\boldsymbol{k}'\neq 0$, $k_3 \neq 0$ and $r>2$, thanks to the Cauchy–Schwarz inequality, we have
\begin{eqnarray*}
&&\hskip-.48in 
B_1 = \sum\limits_{\substack{\boldsymbol{j}'+\boldsymbol{k}'+\boldsymbol{l}'=0\\  k_3+l_3 = 0 \\  j_3, k_3, \boldsymbol{j}' \neq 0}}  C_r \frac{|k_3|}{|j_3|} |\boldsymbol{k}'|^{\frac{1}{2}} |\boldsymbol{j}'|^{\frac{1}{2}}|\boldsymbol{j}|^{r} |\boldsymbol{l}|^{r} e^{\tau |\boldsymbol{j}'|} e^{\gamma |j_3|}e^{\tau |\boldsymbol{k}'|} e^{\gamma |k_3|} e^{\tau |\boldsymbol{l}'|} e^{\gamma |l_3|} |\hat{f}_{\boldsymbol{j}}| |\hat{g}_{\boldsymbol{k}}| |\hat{h}_{\boldsymbol{l}}| \nonumber \\
&&\hskip-.58in 
= C_r \sum\limits_{\substack{\boldsymbol{k}\in \mathbb{Z}^3 \\ \boldsymbol{k}',k_3\neq 0} } |\boldsymbol{k}'|^{\frac{1}{2}} |k_3| e^{\tau |\boldsymbol{k}'|} e^{\gamma |k_3|} |\hat{g}_{\boldsymbol{k}}|  \sum\limits_{\substack{\boldsymbol{j}\in \mathbb{Z}^3 \\ j_3, \boldsymbol{j}' \neq 0} } \frac{1}{|j_3|} |\boldsymbol{j}'|^{\frac{1}{2}} |\boldsymbol{j}|^{r} e^{\tau |\boldsymbol{j}'|} e^{\gamma |j_3|} |\hat{f}_{\boldsymbol{j}}| 
\\
&&\hskip.58in \times|(\boldsymbol{j}'+\boldsymbol{k}',k_3)|^{r}e^{\tau |\boldsymbol{j}'+\boldsymbol{k}'|} e^{\gamma |k_3|}|\hat{h}_{-(\boldsymbol{j}'+\boldsymbol{k}',k_3)}| \nonumber\\
&&\hskip-.58in 
\leq C_r \sum\limits_{\boldsymbol{k}'\in \mathbb{Z}^2} |\boldsymbol{k}'|^{1-r}  \sum\limits_{k_3\in \Z}  |\boldsymbol{k}'|^{r-\frac12} |k_3| |\hat{g}_{\boldsymbol{k}}| e^{\tau |\boldsymbol{k}'|} e^{\gamma |k_3|} \Big( \sum\limits_{\boldsymbol{j}\in \mathbb{Z}^3} |\boldsymbol{j}'||\boldsymbol{j}|^{2r} e^{2\tau |\boldsymbol{j}'|} e^{2\gamma |j_3|}|\hat{f}_{\boldsymbol{j}}|^2\Big)^{\frac{1}{2}}  \nonumber \\
&&\hskip.58in
\times \Big( \sum\limits_{j_3\neq 0} \frac{1}{|j_3|^2} \sum\limits_{\boldsymbol{j}'\in \mathbb{Z}^2 }  |(\boldsymbol{j}'+\boldsymbol{k}',k_3)|^{2r}e^{2\tau |\boldsymbol{j}'+\boldsymbol{k}'|} e^{2\gamma |k_3|}|\hat{h}_{-(\boldsymbol{j}'+\boldsymbol{k}',k_3)}|^2\Big)^{\frac{1}{2}}  \nonumber\\
&&\hskip-.58in \leq C_r \|A_h^{\frac{1}{2}} A^{r} e^{\tau A_h} e^{\gamma A_z} f\| \sum\limits_{\boldsymbol{k}'\in \mathbb{Z}^2} |\boldsymbol{k}'|^{1-r} \Big(\sum\limits_{k_3\in \Z}  |\boldsymbol{k}'| |\boldsymbol{k}|^{2r} |\hat{g}_{\boldsymbol{k}}|^2 e^{2\tau |\boldsymbol{k}'|} e^{2\gamma |k_3|} \Big)^{\frac{1}{2}}  \nonumber \\
&&\hskip.58in
\times \Big( \sum\limits_{k_3\in\Z} \sum\limits_{\boldsymbol{j}'\in \mathbb{Z}^2 }  |(\boldsymbol{j}'+\boldsymbol{k}',k_3)|^{2r}e^{2\tau |\boldsymbol{j}'+\boldsymbol{k}'|} e^{2\gamma |k_3|} |\hat{h}_{-(\boldsymbol{j}'+\boldsymbol{k}',k_3)}|^2\Big)^{\frac{1}{2}}  \nonumber\\
&&\hskip-.58in \leq C_r \|A_h^{\frac{1}{2}} A^{r} e^{\tau A_h} e^{\gamma A_z} f\| \|A^{r} e^{\tau A_h} e^{\gamma A_z} h\| \Big(\sum\limits_{\boldsymbol{k}'\in \mathbb{Z}^2} |\boldsymbol{k}'|^{2-2r}\Big)^{\frac{1}{2}} 
\\
&&\hskip.58in \times\Big( \sum\limits_{\boldsymbol{k}'\in \mathbb{Z}^2} \sum\limits_{k_3\in \Z}    |\boldsymbol{k}'||\boldsymbol{k}|^{2r} |\hat{g}_{\boldsymbol{k}}|^2 e^{2\tau |\boldsymbol{k}'|} e^{2\gamma |k_3|}\Big)^{\frac{1}{2}}  \nonumber \\
&&\hskip-.58in
\leq C_r \|A_h^{\frac{1}{2}} A^{r} e^{\tau A_h} e^{\gamma A_z} f\| \|A_h^{\frac{1}{2}} A^{r} e^{\tau A_h} e^{\gamma A_z} g\| \|A^{r}  e^{\tau A_h} e^{\gamma A_z} h\|.
\end{eqnarray*}
For $B_2$, since $r >2$ and $k_3\neq 0$, thanks to the Cauchy–Schwarz inequality, we have
\begin{eqnarray*}
&&\hskip-.48in 
B_2 = \sum\limits_{\substack{\boldsymbol{j}'+\boldsymbol{k}'+\boldsymbol{l}'=0\\  k_3+l_3 = 0 \\  j_3, k_3, \boldsymbol{j}' \neq 0}}  C_r \frac{|k_3|}{|j_3|} |\boldsymbol{j}'|^{\frac{1}{2}} |\boldsymbol{l}'|^{\frac{1}{2}}|\boldsymbol{j}|^{r} |\boldsymbol{l}|^{r} e^{\tau |\boldsymbol{j}'|} e^{\gamma |j_3|}e^{\tau |\boldsymbol{k}'|} e^{\gamma |k_3|} e^{\tau |\boldsymbol{l}'|} e^{\gamma |l_3|} |\hat{f}_{\boldsymbol{j}}| |\hat{g}_{\boldsymbol{k}}| |\hat{h}_{\boldsymbol{l}}| \nonumber \\
&&\hskip-.58in 
\leq C_r \sum\limits_{\substack{\boldsymbol{k}\in \mathbb{Z}^3 \\ \boldsymbol{k}',k_3\neq 0} }  |\boldsymbol{k}| e^{\tau |\boldsymbol{k}'|} e^{\gamma |k_3|} |\hat{g}_{\boldsymbol{k}}|  \sum\limits_{\substack{\boldsymbol{j}\in \mathbb{Z}^3 \\ j_3, \boldsymbol{j}' \neq 0} } \frac{1}{|j_3|} |\boldsymbol{j}'|^{\frac{1}{2}} |\boldsymbol{j}|^{r} e^{\tau |\boldsymbol{j}'|} e^{\gamma |j_3|} |\hat{f}_{\boldsymbol{j}}|
\nonumber\\
&&\hskip.58in
\times
|\boldsymbol{j}'+\boldsymbol{k}'|^{\frac12} |(\boldsymbol{j}'+\boldsymbol{k}',k_3)|^{r}e^{\tau |\boldsymbol{j}'+\boldsymbol{k}'|} e^{\gamma |k_3|}|\hat{h}_{-(\boldsymbol{j}'+\boldsymbol{k}',k_3)}| \nonumber\\
&&\hskip-.58in 
\leq C_r \sum\limits_{\boldsymbol{k}'\in \mathbb{Z}^2} |(\boldsymbol{k}',\pm 1)|^{1-r}  \sum\limits_{k_3\in \Z}  |\boldsymbol{k}|^{r} | |\hat{g}_{\boldsymbol{k}}| e^{\tau |\boldsymbol{k}'|} e^{\gamma |k_3|} \Big( \sum\limits_{\boldsymbol{j}\in \mathbb{Z}^3} |\boldsymbol{j}'||\boldsymbol{j}|^{2r} e^{2\tau |\boldsymbol{j}'|} e^{2\gamma |j_3|}|\hat{f}_{\boldsymbol{j}}|^2\Big)^{\frac{1}{2}}  \nonumber \\
&&\hskip-.38in
\times \Big( \sum\limits_{j_3\neq 0} \frac{1}{|j_3|^2} \sum\limits_{\boldsymbol{j}'\in \mathbb{Z}^2 } |\boldsymbol{j}'+\boldsymbol{k}'|  |(\boldsymbol{j}'+\boldsymbol{k}',k_3)|^{2r}e^{2\tau |\boldsymbol{j}'+\boldsymbol{k}'|} e^{2\gamma |k_3|}|\hat{h}_{-(\boldsymbol{j}'+\boldsymbol{k}',k_3)}|^2\Big)^{\frac{1}{2}}  \nonumber\\
&&\hskip-.58in \leq C_r \|A_h^{\frac{1}{2}} A^{r} e^{\tau A_h} e^{\gamma A_z} f\| \sum\limits_{\boldsymbol{k}'\in \mathbb{Z}^2} |(\boldsymbol{k}',\pm 1)|^{1-r} \Big(\sum\limits_{k_3\in \Z} |\boldsymbol{k}|^{2r} |\hat{g}_{\boldsymbol{k}}|^2 e^{2\tau |\boldsymbol{k}'|} e^{2\gamma |k_3|} \Big)^{\frac{1}{2}}  \nonumber \\
&&\hskip-.38in
\times \Big( \sum\limits_{k_3\in\Z} \sum\limits_{\boldsymbol{j}'\in \mathbb{Z}^2 } |\boldsymbol{j}'+\boldsymbol{k}'| |(\boldsymbol{j}'+\boldsymbol{k}',k_3)|^{2r}e^{2\tau |\boldsymbol{j}'+\boldsymbol{k}'|} e^{2\gamma |k_3|} |\hat{h}_{-(\boldsymbol{j}'+\boldsymbol{k}',k_3)}|^2\Big)^{\frac{1}{2}}  \nonumber\\
&&\hskip-.58in \leq C_r \|A_h^{\frac{1}{2}} A^{r} e^{\tau A_h} e^{\gamma A_z} f\| \|A_h^{\frac{1}{2}} A^{r} e^{\tau A_h} e^{\gamma A_z} h\| \Big(\sum\limits_{\boldsymbol{k}'\in \mathbb{Z}^2} |(\boldsymbol{k}',\pm 1)|^{2-2r}\Big)^{\frac{1}{2}} 
\\
&&\hskip.58in
\times\Big( \sum\limits_{\boldsymbol{k}'\in \mathbb{Z}^2} \sum\limits_{k_3\in \Z}    |\boldsymbol{k}|^{2r} |\hat{g}_{\boldsymbol{k}}|^2 e^{2\tau |\boldsymbol{k}'|} e^{2\gamma |k_3|}\Big)^{\frac{1}{2}}  \nonumber \\
&&\hskip-.58in
\leq C_r \|A_h^{\frac{1}{2}} A^{r} e^{\tau A_h} e^{\gamma A_z} f\| \| A^{r} e^{\tau A_h} e^{\gamma A_z} g\| \|A_h^{\frac{1}{2}} A^{r}  e^{\tau A_h} e^{\gamma A_z} h\|,
\end{eqnarray*}
where we have used that $|\boldsymbol{k}|^{1-r} \leq |(\boldsymbol{k}',\pm 1)|^{1-r}$ when $k_3\neq 0$.
For $B_3$, since $r>2$, thanks to the Cauchy–Schwarz inequality, we have
\begin{eqnarray*}
&&\hskip-.48in 
B_3 = \sum\limits_{\substack{\boldsymbol{j}'+\boldsymbol{k}'+\boldsymbol{l}'=0\\  k_3+l_3 = 0 \\  j_3, k_3, \boldsymbol{j}' \neq 0}}  C_r \frac{|k_3|}{|j_3|} |\boldsymbol{k}'|^{\frac{1}{2}} |\boldsymbol{j}'|^{\frac{1}{2}}|\boldsymbol{k}|^{r} |\boldsymbol{l}|^{r} e^{\tau |\boldsymbol{j}'|} e^{\gamma |j_3|}e^{\tau |\boldsymbol{k}'|} e^{\gamma |k_3|} e^{\tau |\boldsymbol{l}'|} e^{\gamma |l_3|} |\hat{f}_{\boldsymbol{j}}| |\hat{g}_{\boldsymbol{k}}| |\hat{h}_{\boldsymbol{l}}| \nonumber \\
&&\hskip-.58in 
= C_r \sum\limits_{\substack{\boldsymbol{j}\in \mathbb{Z}^3 \\ j_3, \boldsymbol{j}'\neq 0} } \frac{1}{|j_3|} |\boldsymbol{j}'|^{\frac{1}{2}}  e^{\tau |\boldsymbol{j}'|} e^{\gamma |j_3|} |\hat{f}_{\boldsymbol{j}}|  \sum\limits_{\substack{\boldsymbol{k}\in \mathbb{Z}^3 \\ k_3, \boldsymbol{k}' \neq 0} }  |\boldsymbol{k}'|^{\frac{1}{2}} |\boldsymbol{k}|^{r} e^{\tau |\boldsymbol{k}'|} e^{\gamma |k_3|} |\hat{g}_{\boldsymbol{k}}| 
\\
&&\hskip.58in \times|(\boldsymbol{j}'+\boldsymbol{k}',k_3)|^{r} |k_3| e^{\tau |\boldsymbol{j}'+\boldsymbol{k}'|} e^{\gamma |k_3|}|\hat{h}_{-(\boldsymbol{j}'+\boldsymbol{k}',k_3)}| \nonumber\\
&&\hskip-.58in 
\leq C_r \sum\limits_{\substack{\boldsymbol{j}\in \mathbb{Z}^3 \\ j_3, \boldsymbol{j}'\neq 0} } |\boldsymbol{j}'|^{\frac12-r}    \frac{1}{|j_3|} |\boldsymbol{j}|^{r} |\hat{f}_{\boldsymbol{j}}| e^{\tau |\boldsymbol{j}'|} e^{\gamma |j_3|} \Big( \sum\limits_{\boldsymbol{k}\in \mathbb{Z}^3} |\boldsymbol{k}'||\boldsymbol{k}|^{2r} e^{2\tau |\boldsymbol{k}'|} e^{2\gamma |k_3|}|\hat{g}_{\boldsymbol{k}}|^2\Big)^{\frac{1}{2}}  \nonumber \\
&&\hskip.58in
\times \Big( \sum\limits_{\boldsymbol{k}\in \mathbb{Z}^3 } |k_3|^2 |(\boldsymbol{j}'+\boldsymbol{k}',k_3)|^{2r}e^{2\tau |\boldsymbol{j}'+\boldsymbol{k}'|} e^{2\gamma |k_3|}|\hat{h}_{-(\boldsymbol{j}'+\boldsymbol{k}',k_3)}|^2\Big)^{\frac{1}{2}}  \nonumber\\
&&\hskip-.58in \leq C_r \|A_h^{\frac{1}{2}} A^{r} e^{\tau A_h} e^{\gamma A_z} g\| \|A_z A^{r} e^{\tau A_h} e^{\gamma A_z} h\| \Big(\sum\limits_{\boldsymbol{j}'\neq 0} \sum\limits_{j_3 \neq 0}  |\boldsymbol{j}'|^{1-2r} \frac1{|j_3|^2}\Big)^{\frac{1}{2}} 
\\
&&\hskip.58in \times\Big( \sum\limits_{\boldsymbol{j}\in \mathbb{Z}^3}   |\boldsymbol{j}|^{2r} |\hat{f}_{\boldsymbol{j}}|^2 e^{2\tau |\boldsymbol{j}'|} e^{2\gamma |j_3|}\Big)^{\frac{1}{2}}  \nonumber \\
&&\hskip-.58in
\leq C_r \| A^{r} e^{\tau A_h} e^{\gamma A_z} f\| \|A_h^{\frac{1}{2}} A^{r} e^{\tau A_h} e^{\gamma A_z} g\| \|A_z A^{r} e^{\tau A_h} e^{\gamma A_z} h\|.
\end{eqnarray*}
The estimate of $B_4$ is similar to that of $B_3$, and one can get 
\begin{equation*}
    B_4 \leq C_r \| A^{r} e^{\tau A_h} e^{\gamma A_z} f\| \|A_z A^{r} e^{\tau A_h} e^{\gamma A_z} g\| \|A_h^{\frac{1}{2}} A^{r} e^{\tau A_h} e^{\gamma A_z} h\|.
\end{equation*}
Therefore, we have
\begin{equation*}
    \begin{split}
        I_{22} \leq &C_r \|A_h^{\frac{1}{2}} A^{r} e^{\tau A_h} e^{\gamma A_z} f\| \left( \|A_h^{\frac{1}{2}} A^{r} e^{\tau A_h} e^{\gamma A_z} g\| \|A^{r}  e^{\tau A_h} e^{\gamma A_z} h\| + \| A^{r} e^{\tau A_h} e^{\gamma A_z} g\| \|A_h^{\frac{1}{2}} A^{r}  e^{\tau A_h} e^{\gamma A_z} h\|\right)
        \\
        + &C_r \| A^{r} e^{\tau A_h} e^{\gamma A_z} f\| \left(\|A_h^{\frac{1}{2}} A^{r} e^{\tau A_h} e^{\gamma A_z} g\| \|A_z A^{r} e^{\tau A_h} e^{\gamma A_z} h\|+\|A_z A^{r} e^{\tau A_h} e^{\gamma A_z} g\| \|A_h^{\frac{1}{2}} A^{r} e^{\tau A_h} e^{\gamma A_z} h\| \right) .
    \end{split}
\end{equation*}
The estimates of $I_{21}$ follow similarly to those of $I_{22}$, and the result is the same. Thus,
\begin{equation*}
    \begin{split}
        I_{2} \leq &C_r \|A_h^{\frac{1}{2}} A^{r} e^{\tau A_h} e^{\gamma A_z} f\| \left( \|A_h^{\frac{1}{2}} A^{r} e^{\tau A_h} e^{\gamma A_z} g\| \|A^{r}  e^{\tau A_h} e^{\gamma A_z} h\| + \| A^{r} e^{\tau A_h} e^{\gamma A_z} g\| \|A_h^{\frac{1}{2}} A^{r}  e^{\tau A_h} e^{\gamma A_z} h\|\right)
        \\
        + &C_r \| A^{r} e^{\tau A_h} e^{\gamma A_z} f\| \left(\|A_h^{\frac{1}{2}} A^{r} e^{\tau A_h} e^{\gamma A_z} g\| \|A_z A^{r} e^{\tau A_h} e^{\gamma A_z} h\|+\|A_z A^{r} e^{\tau A_h} e^{\gamma A_z} g\| \|A_h^{\frac{1}{2}} A^{r} e^{\tau A_h} e^{\gamma A_z} h\| \right).
    \end{split}
\end{equation*}
Combining the estimates of $I_1$ and $I_2$ yields
\begin{equation*}
    \begin{split}
       &\Big|\Big\langle A^r e^{\tau A_h} e^{\gamma A_z} Q(f,g), A^r e^{\tau A_h} e^{\gamma A_z} h  \Big\rangle\Big| 
       \\
       \leq &C_r\Big[\|f\|_{\tau,r,\gamma,r} \|A_h^{\frac12} A^r e^{\tau A_h} e^{\gamma A_z} g\| \|A_h^{\frac12} A^r e^{\tau A_h} e^{\gamma A_z} h\|
     \\
     & \quad+ \|g\|_{\tau,r,\gamma,r} \|A_h^{\frac12} A^r e^{\tau A_h} e^{\gamma A_z} f\|  \|A_h^{\frac12} A^r e^{\tau A_h} e^{\gamma A_z} h\|
     \\
     & \quad+ \|h\|_{\tau,r,\gamma,r} \|A_h^{\frac12} A^r e^{\tau A_h} e^{\gamma A_z} f\|  \|A_h^{\frac12} A^r e^{\tau A_h} e^{\gamma A_z} g\|\Big]
     \\
     &+ C_r \| A^{r} e^{\tau A_h} e^{\gamma A_z} f\| \left(\|A_h^{\frac{1}{2}} A^{r} e^{\tau A_h} e^{\gamma A_z} g\| \|A_z A^{r} e^{\tau A_h} e^{\gamma A_z} h\|+\|A_z A^{r} e^{\tau A_h} e^{\gamma A_z} g\| \|A_h^{\frac{1}{2}} A^{r} e^{\tau A_h} e^{\gamma A_z} h\| \right).
    \end{split}
\end{equation*}

\end{proof}

The next lemma comes from \cite[Lemma 2.1]{foias1989gevrey}, addressing an important property of the space $\mathcal{D}_{\tau,r}$. 
\begin{lemma} \label{lemma-banach-algebra}
If $\tau \geq 0$ and $r>\frac{3}{2}$, then $\mathcal{D}_{\tau,r}$ is a Banach algebra, and for any $f, g\in \mathcal{D}_{\tau,r}$, we have 
\begin{eqnarray*}
\| fg\|_{\tau,r} \leq C_{r}\|  f\|_{\tau,r} \|  g\|_{\tau,r}.
\end{eqnarray*}
For the semi-norm, we also have a similar estimate
\begin{eqnarray*}
\| A^r e^{\tau A} (fg)\| \leq C_{r}  \Big(|\hat{f}_0|+\| A^r e^{\tau A} f\|\Big) \Big(|\hat{g}_0|+ \|  A^r e^{\tau A} g\|\Big).
\end{eqnarray*}
\end{lemma}

Similar to Lemma~\ref{lemma-banach-algebra}, one has the follwing lemma concerning the space $\mathcal{D}_{\tau,r,\gamma,r}$.
\begin{lemma} \label{lemma-banach-algebra-viscous}
If $\tau,\gamma \geq 0$ and $r>\frac{3}{2}$, then $\mathcal{D}_{\tau,r,\gamma,r}$ is a Banach algebra, and for any $f, g\in \mathcal{D}_{\tau,r,\gamma,r}$, we have 
\begin{eqnarray*}
\| fg\|_{\tau,r,\gamma,r} \leq C_{r}\|  f\|_{\tau,r,\gamma,r} \|  g\|_{\tau,r,\gamma,r}.
\end{eqnarray*}
For the semi-norm, we also have a similar estimate
\begin{eqnarray*}
\| A^r e^{\tau A_h} e^{\gamma A_z} (fg)\| \leq C_{r}  \Big(|\hat{f}_0|+\| A^r e^{\tau A_h} e^{\gamma A_z} f\|\Big) \Big(|\hat{g}_0|+ \|  A^r e^{\tau A_h} e^{\gamma A_z} g\|\Big).
\end{eqnarray*}
\end{lemma}

\bibliographystyle{plain}
\bibliography{Reference}

\end{document}